\documentclass{m2an}
\usepackage{graphicx}
\usepackage{camnum}
\usepackage{amsmath}

\usepackage{tikz,pgfplots}
\pgfplotsset{compat=1.12}

\allowdisplaybreaks

\newcommand{\ii}{{\mathrm i}}

\usepackage{a4}

\def\beq{\begin{equation}}
\def\eeq{\end{equation}}
\def\beqnn{\begin{eqnarray*}}
\def\eeqnn{\end{eqnarray*}}
\def\bea{\begin{eqnarray}}
\def\eea{\end{eqnarray}}

\newcommand{\e}{\ensuremath{\mathrm{e}}}

\newcommand{\yy}{{\bf y}}

\begin{document}

\title{Positivity-preserving methods for ordinary differential equations}
\author{Sergio Blanes}\address{Instituto de Matem\'atica Multidisciplinar,
  Universitat Polit\`ecnica de Val\`encia,
  E-46022  Valencia, \textsc{Spain}.
  Email: \texttt{serblaza@imm.upv.es}}
\author{Arieh Iserles}\address{Department of Applied Mathematics and Theoretical Physics,
  Centre for Mathematical Sciences,
  University of Cambridge,
  Wilberforce Road, Cambridge CB4 1LE,
  \textsc{United Kingdom}.
  Email: \texttt{ai10@cam.ac.uk}}
\author{Shev Macnamara}\address{Australian Research Council Centre of Excellence,
  for Mathematical and Statistical Frontiers (ACEMS),
  School of Mathematical and Physical Sciences,
  University of Technology Sydney, NSW~2007,
  \textsc{Australia}.
  Email: \texttt{shev.macnamara@uts.edu.au}}
\date{}
%


\begin{abstract}
Many important applications are modelled by differential equations with positive solutions.
However, it remains an outstanding open problem to develop numerical methods that are both (i) of a high order of accuracy and (ii) capable of preserving positivity.
It is known that the two main families of numerical methods, Runge--Kutta methods and multistep methods,  face an order barrier. If they preserve positivity, then they are constrained to low accuracy: they cannot be better than first order.
We propose novel methods that overcome this barrier: 
second order methods that preserve positivity unconditionally and a third order method that preserves positivity under very mild conditions.
%
%
Our methods apply to a large class of differential equations that have a special graph Laplacian structure, which we elucidate.
The equations need  be neither linear nor autonomous and the graph Laplacian need not be symmetric.
This algebraic structure arises naturally in many important applications where positivity is required.
We showcase our new methods on applications where standard high order methods fail to preserve positivity, including infectious diseases, Markov processes, master equations and chemical reactions.
\end{abstract}

\subjclass{65L05, 65P99, 65L04}
\keywords{Positivity-preserving methods, graph Laplacian matrices, exponential integrators, Magnus integrators}
\maketitle

\section{Introduction}  

{{}

Numerical integration of mathematical models is an essential step in the implementation and analysis of population models: chemical reactions (see for example  \cite{edsberg74ipc,sandu01pni} or \cite{hairer10sod}),  biochemical systems \cite{bruggeman07aso}, and the evolution of epidemics \cite{kermack27act} (see also \cite{giordano20mtc} and references therein). 
Such models are usually formulated as a system of Ordinary Differential Equations (ODEs)
\begin{equation}
\label{eq:ODE1}
  {\bf y}' = {\bf f}(t,{\bf y}), \qquad {\bf y}(0)={\bf y}_0\in\BB{R}^d,
\end{equation}
where ${\bf f}(t,{\bf y})$ is, in the context of this paper,  consistent with two requirements of the application being modeled.
First,  
 if $y_i^0\geq 0, \ i=1,\ldots,d$ then we have {\em positivity preservation\/}:
\begin{displaymath}
 y_i(t)\geq 0 \ \ \forall t, \quad i=1,\ldots,d.
\end{displaymath}

Second,
{{}{there exist ${\bf w}_{\ell}=(w_{\ell,1},\ldots,w_{\ell,d})^\top, \ \ell=1,2,\ldots,k$ such that ${\bf w}_{\ell}^\top {\bf f}(t,{\bf y})=0$ so,}
 the solution satisfies the conditions (with ${\bf y}=(y_1,\ldots,y_d)^\top$, ${\bf y}_0=(y_1^0,\ldots,y_d^0)^\top$)
\begin{displaymath}
  \sum_{i=1}^d w_{\ell,i}y_i=\sum_{i=1}^d w_{\ell,i}y_i^0=c_{\ell},  
	\qquad \ell=1,2,\ldots,k
\end{displaymath}
with $w_{\ell,i}\geq 0$ and $c_{\ell}>0$. 
The most important special case is $k=1$ and ${\bf w}_1
=(1,\ldots,1)={\bf 1}$, which is referred to as {\em mass preservation\/}, and in this case we may assume without loss of generality that $c_1=1$. 

Although the focus of this article is mainly on 
positivity and mass preservation
ODEs, positivity preservation is a  much wider challenge. 
For example, 
{{}{Lotka-Volterra models \cite{beck2015otp,diele2020gni} preserve positivity but not mass as well as some parabolic problems \cite{hansen2012aso}}. 
The stochastic differential equation  associated with the Nobel prize winning Black--Scholes model in finance has positive solutions, but  standard numerical solvers, such as the Euler--Maruyama method, fail to preserve  positivity. 
The Kolmogorov Lecture at the Ninth World Congress in Probability and Statistics concerned methods for preserving positivity in the setting of the stochastic Langevin equations \cite{leite2019constrained}.

{We note in passing that even with these two requirements, \eqref{eq:ODE1} may display rich dynamical behaviour: some systems of this kind converge to a unique steady state, others have a number of steady states, yet others exhibit oscillatory behaviour. 

{{}{The methods proposed in this work are constructed to preserve positivity, while keeping linear invariants preservation to high accuracy (symplectic integrators preserve the symplectic structure of Hamiltonian systems while not exactly preserving energy, but this gives good properties in regards to error propagation over long time intervals). 
However, in the case where only mass preservation is required 
there are well known mathematical results that allow us to adapt the methods to preserve exactly, and for this reason this case is now treated in more detail.}

\subsection{Graph Laplacians and ODEs}

A useful way to envisage mass and positivity preservation is that for every $t\geq0$  the state variable ${\bf y}(t)$ is a discrete probability distribution of $d$ species.  
This corresponds to the case $k=1$, ${\bf w}_1={\bf 1}$
and, as we will show in Proposition \ref{prop:GraphLaplacians}, these properties can be preserved if the vector field in \eqref{eq:ODE1} can be written in the form (see also e.g. \cite{bertolazzi1996pac,colonna2020otr,formaggia2011pac})
\begin{displaymath}
    {\bf f}(t,{\bf y}) = A(t,{\bf y}){\bf y}
\end{displaymath}
where  the matrix $A:\BB{R}\times \BB{R}^d\rightarrow\BB{R}^{d\times d}$ is a graph Laplacian.


\vspace{6pt}
\noindent{\bf Definition} \textit{An $n\times n$ real matrix $A$ is a graph Laplacian if it has  the following properties:
\begin{description}
\item[Property 1 (pattern of signs)] 
$A_{k,\ell}\geq0$ for $k,\ell=1,\ldots,n$, $k\neq\ell$,  $A_{k,k}\leq0$ for $k=1,\ldots,n$ and 
\item[Property 2 (zero column sum)] 
$\sum_{k=1}^n A_{k,\ell}=0$ for $\ell=1,\ldots,n$.
\end{description} }

\vspace{6pt}
We denote the set of all $n\times n$ graph Laplacians by ${\cal L}_n$.
The same term `graph Laplacian' is used with different meanings in the literature -- in our work, we allow it to be non-symmetric.

For simplicity, we consider the autonomous case. 
(The general nonautonomous case can be considered similarly, as we will show latter.)
We  focus on the solution of the nonlinear ODE
\begin{equation}
\label{eq:ODE2}
  {\bf y}' = A({\bf y}){\bf y}, \qquad \qquad {\bf y}(0)={\bf y}_0\in\BB{R}^d,
\end{equation}
{{}{where we assume throughout that  $A({\bf y})$ has the same pattern of signs as a graph Laplacian, i.e. we assume Property 1 of the definition above.
(In some examples, such as the MAPK cascade example, we do not assume Property 2, i.e.\ we do not always assume ${\bf 1}^\top A({\bf y})={\bf 0}^\top$, and we demonstrate that our methods can nevertheless work well.)
We typically also assume that all components of the initial condition are  nonnegative.}
Many applications fit this framework:
Markov processes in continuous time on discrete states; master equations \cite{macnamara2008multiscale}; single molecule chemistry  \cite{DimpleMachineCohen2012} (Fig.~4); studies of robustness of Turing pattern formation in stochastic settings  \cite{maini2012turing,hellander2017robustness}; and lasers and quantum dots \cite{timm2009random}.

Given two compatible matrices $P$ and $Q$ we say that $P\succ Q$ if $P_{i,j}>Q_{i,j}$ for all $i,j$ and $P\succeq Q$ if $P_{i,j}\geq Q_{i,j}$.
We assume that ${\bf y}_0\succeq{\bf 0}$ and ${\bf 1}^\top{\bf y}_0=1$. Then the solutions of \eqref{eq:ODE2} have the following desirable features.

\begin{proposition}
\label{prop:GraphLaplacians}
Solutions of \eqref{eq:ODE2} 
{{}{with ${\bf y}_0\succeq{\bf 0}, {\bf 1}^T{\bf y}_0=1$}}
 have the following two properties:
\begin{description}
 \item[Positivity] ${\bf y}(t)\succeq{\bf 0}$ for all $t\geq0$, and
\item[Conservation of mass]
  ${\bf 1}^\top\!{\bf y}(t)=1$, for all $t\geq0$. 
 \end{description}

\end{proposition}

\begin{proof}
  The statement about mass conservations is trivial, because
  \begin{displaymath}
    {\bf 1}^\top{\bf y}'(t)={\bf 1}^\top A({\bf y}(t)){\bf y}(t)={\bf 0}^\top{\bf y}(t)=0
  \end{displaymath}
  implies that ${\bf 1}^\top{\bf y}(t)\equiv\mbox{const}={\bf 1}^\top{\bf y}_0=1$. 

  To prove the statement about positivity, we consider any $t^*\geq 0$ such that there exists $k^*\in\{1,2,\ldots,d\}$ with $y_{k^*}(t^*)=0$ and such that ${\bf y}(t^*)\succeq{\bf 0}$ -- clearly, unless such $t^*$ exists, ${\bf y}(t)$ stays forever  in the nonnegative cone. Note that it is perfectly possible for $t^*$ to be zero, also it is possible that several components of ${\bf y}(t)$ vanish at $t=t^*$, this makes no difference to our argument. We note that, by \eqref{eq:ODE2},
  \begin{displaymath}
    y_{k^*}'(t^*)=\sum_{\ell=1}^d A_{k^*,\ell}({\bf y}(t^*))y_\ell(t^*)\geq0,
  \end{displaymath}
  because $A$ is a graph Laplacian, so off-diagonal entries are nonnegative. Therefore $y_{k^*}$ cannot change sign at $t^*$, and it must stay in the nonnegative cone.
\end{proof}

\paragraph{Remark.}
Note in the proof of Proposition~\ref{prop:GraphLaplacians} that 
property 1 alone of the definition of the Laplacian (pattern of signs) suffices to give positivity, and that, separately, 
property 2 alone of the definition of the Laplacian suffices to give mass preservation.
In particular, if the matrix $A({\bf y})$ has the same pattern of $\pm$ signs as a Laplacian (but we make no assumption on the column sums of $A({\bf y})$), then it is still true that solutions of $ {\bf y}' = A({\bf y}){\bf y}$, preserve positivity.


Let us now consider some properties of graph Laplacian matrices that  allow us to deduce additional qualitative properties of the solution of \eqref{eq:ODE2}.

\begin{theorem}
  \label{thm:GEvals}
  Let $A\in{\cal L}_n$. Then it has an  eigenvalue at the origin, which is simple if $A$ is irreducible, and all its other eigenvalues reside in $\BB{C}^-=\{z\in\BB{C}\,:\, \Re z<0\}$.
\end{theorem}

\begin{proof}
  Since ${\bf 1}^\top\! A={\bf 0}^\top$, it follows that $0\in\sigma(A)$. To locate the remaining eigenvalues we use the Gerschgorin theorem, applying it to columns (typically it is applied to rows, but this makes no difference). Thus, letting
  \begin{displaymath}
    \BB{S}_\ell=\left\{z\in C\,:\, |z-A_{\ell,\ell}|\leq \sum_{k\neq \ell} |A_{k,\ell}|\right\},\qquad \ell=1,\ldots,n,
  \end{displaymath}
  we have $\sigma(A)\subset\bigcup_{\ell=1}^n \BB{S}_\ell$. By the definition of graph Laplacian, all Gerschgorin discs live in ${\bf cl}\,\BB{C}^-$ and adjoin $ \ii \BB{R}$ only at the origin. Therefore $\sigma(A)\setminus\{0\}\in\BB{C}^-$.  
  
  It remains to prove that 0 is a simple eigenvalue. Let $\alpha=\min_{k=1,\ldots,n} A_{k,k}$, then the entries of $B=A-\alpha I\neq O$ are all nonnegative. Therefore, according to Frobenius--Perron theory \cite{berman79nmm}, irreducibility implies that the largest in modulus eigenvalue of $B$ is positive and simple. Since this is $-\alpha$, it follows that 0 is a simple eigenvalue of $A$.
\end{proof}

Incidentally, one of the less well-known formulations of the Gerschgorin theorem  states that if $A$ is irreducible then an eigenvalue might be on the boundary of one Gerschgorin disc only if it is on the boundary of all Gerschgorin discs -- this is certainly the case with 0.


\begin{proposition}
  \label{prop:Evalues}
Assume the matrix $A\in{\cal L}_n$ is symmetric.
 Then $\D\|{\bf y}(t)\|^2/\D t\leq0$.
\end{proposition}

\begin{proof}
  We compute
  \begin{displaymath}
    \frac12 \frac{\D\|{\bf y}(t)\|^2}{\D t}={\bf y}^\top(t){\bf y}'(t)={\bf y}^\top(t) A({\bf y}(t)){\bf y}(t)\leq \alpha_+(A({\bf y}(t)))\|{\bf y}(t)\|^2,
  \end{displaymath}
  where $\alpha_+(B)$ is the {\em spectral abscissa\/} -- the eigenvalue of the matrix $B$ with the largest real part (which in the case of $A$ is real because of the  Perron--Frobenius theory). This is true because $\alpha_+(B)\geq {\bf v}^\top B{\bf v}/\|{\bf v}\|^2$ for any square matrix $B$ and a nonzero vector ${\bf v}$. Since our $A({\bf y})$ is graph Laplacian, it follows at once from the Gerschgorin theorem that $\alpha_+(A({\bf y}(t)))\leq 0$ and, since $0\in\sigma(A({\bf y}(t)))$, we deduce that $\D\|{\bf y}(t)\|^2/\D t\leq0$.
\end{proof} 

Let $\hat{\bf y}$ be the eigenvector corresponding to the simple eigenvalue $0$. In the symmetric case, it is  clear that $\|\yy(t)-\hat{\yy}\|^2$ is a monotonically decreasing function -- using the fact that $A({\bf y}(t))\hat{{\bf y}}={\bf 0}$,
\begin{displaymath}
  [\yy(t)-\hat{\yy}]'=\yy'(t)=A(\yy(t))\yy(t)=A(\yy(t))[{\bf y}(t)-\hat{{\bf y}}]
\end{displaymath}
and we continue as before. 
 
 In the nonsymmetric case, the issue of stability needs more discussion.
 The two defining properties of the graph Laplacian together ensure that the columns of the matrix exponential are probability vectors, so that, when $A$ is a constant matrix, in the 1-norm we always have $\|\exp(tA)\| =1$, $t\geq0$.
 In the case of a constant matrix, these matrices are sometimes known as `W-matrices' in the statistical physics literature and, by studying the adjoint  ${\bf z}'(t) = A^\top {\bf z}(t)$ -- with arguments similar to those of our Proposition~\ref{prop:GraphLaplacians} -- it is known that the minimum of the solution ${\bf z}$ is increasing, and that the maximum is decreasing.
  In the 2-norm, a sufficient condition for strong stability of ${\bf y}' = A {\bf y}(t)$ with solution ${\bf y}(t) = \exp(tA){\bf y}(0)$, is that $(A+A^\top)$ be negative definite. 
 Note that this condition is more restrictive than merely the assumption that the eigenvalues of $A$ have negative real part
  (because then it would still be possible that $(A+A^\top)$ had a positive eigenvalue).
 This issue of stability  is related to \textit{`the hump'} in the classical literature on the numerical analysis of the matrix exponential, and to the \textit{lognorm}, and also to the subject of \textit{pseudospectra.}
  Nonsymmetric graph Laplacians exhibit significant pseudospectra, manifesting themselves in various ways, such as a more subtle stability analysis, and the failure of  standard eigenvalue algorithms \cite{iserles2019applications,ShevCauchyIntegralMasterEqnPseudoSpectraCTAC2015,macnamara20sob}.
  A sufficient condition for stability of operator splitting methods is that each part separately be strongly stable, although this may be too pessimistic in practice.
For operator splitting methods, the graph Laplacian can sometimes be expressed as the sum of two matrices, each of which is separately a graph Laplacian
with a physical interpretation \cite{MacBer08}.
In general, operator splitting
does not preserve the steady state \cite{speth2013balanced} -- so it is worth pointing out that the novel splitting methods that we introduce in this work, for example later in \eqref{eq:smoothing}, in our numerical experiments, do have the desirable property that they preserve the steady state.
In the nonautonomous case, but still linear case,  it can be shown under suitable assumptions that the difference of any two solutions is decreasing in the 1-norm, but the issue of stability is \textit{much} more delicate.
For instance, see the catalogue of counterexamples, and Theorem 3.1 described in   \cite{earnshaw2010global}.


To sum up, the solution of a mathematical model given by \eqref{eq:ODE2} with ${\bf y}_0\succeq {\bf 0}$ and where $A({\bf y})$ is a graph Laplacian matrix (assuming  ${\bf y}\succeq {\bf 0}$) always preserves mass and always preserves positivity.
Often, the model \eqref{eq:ODE2} is also stable and converges to a steady state. 
These features correspond to the phenomenological desiderata in for example epidemiological models.

In theory, there are always exact formulae for the right eigenvector  corresponding to the zero eigenvalue of a nonsymmetric graph Laplacian matrix $A$, via the Matrix-Tree Theorem \cite{gunawardena2012linear}.
This is the steady state of the corresponding linear Laplacian dynamical system, and in special cases, there are also formulae for the dynamical solutions
\cite{earnshaw2010global,earnshaw2010invariant,iserles2019applications}.

Unfortunately, in general, the exact solution of these dynamical systems is  unknown, so we need to resort to numerical algorithms.
Using backward error analysis, we can envisage a numerical method as the exact solution of a perturbed model. While this is typically adequate across a single step, unless the  method is  chosen carefully, a numerical solution is highly unlikely to respect the important special structure of \eqref{eq:ODE2} across the entire time interval of interest.


The mathematical models we are considering in this paper are based on differential equations whose solutions preserve some underlying geometric structure. The design and analysis of numerical integrators that preserve the qualitative features  of the underlying differential equations is the subject of  {\em Geometric Numerical Integration\/} \cite{blanes16aci,hairer10gni,iserles09afc,sanzserna94nhp}. We are not only concerned with the accuracy and stability of  numerical schemes but also with their geometric properties, which reflect important features of the phenomena being modelled. This endows the integrators with an improved qualitative behaviour, but also typically leads to significantly more accurate results.


For example, in \cite{giordano20mtc} the authors consider a mathematical model for the COVID-19 epidemic in Italy, while paying much attention so that the proposed model has the structure of \eqref{eq:ODE2}, but then numerically solve it using the first order explicit Euler method 
\begin{displaymath}
  {\bf y}_{n+1} = {\bf y}_{n} + h {\bf f}({\bf y}_{n})
\end{displaymath}
where $h$ is the time step and ${\bf y}_{n}\simeq {\bf y}(t_{n})$ with $t_n=t_0+nh$. 
%
We easily see that
\begin{displaymath}
  {\bf 1}^\top{\bf y}_{n+1} = {\bf 1}^\top{\bf y}_{n} + h {\bf 1}^\top{\bf f}({\bf y}_{n}) = {\bf 1}^\top{\bf y}_{n} = \ldots =  {\bf 1}^\top{\bf y}_{0} 
\end{displaymath}
and then the mass is preserved (this is also the case for most standard methods like Runge--Kutta or multistep methods).
{{}{However, it is well known that, in general, 
this method does not preserve positivity unconditionally.}


This inadequate behaviour cannot be rectified by a standard higher-order method: in \cite{bolley78cdl} it is shown that within the class of linear multistep and Runge--Kutta methods unconditional positivity restricts the order of the method to just one.\footnote{This is a necessary condition which, alas, is not sufficient: the above explicit Euler method is of order one but does not preserve positivity.}

For non-stiff problems and for relatively short time integration, an Euler method, or any other standard method, can provide sufficiently accurate, satisfactory results. However, if a mathematical model is stiff  (this is typical to equations of chemical kinetics) or need be solved for long time intervals,  standard methods may produce negative solutions or become unstable. While the stiffness in chemical kinetics equations can be dealt with using implicit methods and mass is preserved by most numerical methods,  positivity remains an outstanding challenge. 

The most efficient solvers considered in \cite{hairer10sod} for low to medium accuracy in the numerical solution of stiff kinetic equations are  Rosenbrock methods. In addition, they are among the simplest implicit schemes to be implemented in a code, yet they fail to preserve positivity. Note that there exist exponential Rosenbrock-type methods \cite{hochbruck09ert} that involve the computation of the exponential of the Jacobian. However, in general, this Jacobian is not a graph Laplacian and positivity cannot be guaranteed. 

The objective of preserving mass and positivity in numerical integration, in particular within the context of chemical kinetics, received  a measure of attention, although perhaps less than it deserves given its importance in applications. An obvious device to avoid the solution from becoming negative is {\em clipping:\/} the practice of converting a negative component to zero. This, of course, interferes with the preservation of mass but the latter can be recovered using laborious optimization  procedure in every time step \cite{sandu01pni}. The effects of this costly algorithm on long-term accuracy and stability are unknown. 

{{}{
Another approach toward preservation of mass and positivity are  Runge--Kutta--Patankar methods \cite{burchard03hoc,kopecz18ooc,patankar80nht,bertolazzi1996pac}.
The idea is to adapt Runge--Kutta-like methods for {\em pro\-duc\-tion--destruction systems\/} in chemical kinetics. 
We will show that this class of methods can be seen as particular approximation to the methods proposed in the present work.}

%


\section{Illustrative examples}

To illustrate our analysis we consider several simple population  models from the literature.

\

\paragraph{Example 1: The Robertson reaction.}
Let us consider the following example of chemical reactions, $A    \longrightarrow  B $ and $ B + B  \longrightarrow   B+C   \longrightarrow  A+C $, leading to the stiff differential equations for concentrations ${\bf y}=(y_1,y_2,y_3)$ of $A,B,C$ \cite{hairer10sod} (P.~157):
\begin{equation}\label{eq:Robertson}
  \begin{array} {rcrrrl}
     y_1' &\!\!\!=\!\!\!& -0.04 y_1&+10^4y_2y_3, & & \qquad y_1(0)=1\\
     y_2' &\!\!\!=\!\!\!&  0.04 y_1&-10^4y_2y_3 & - 3\cdot 10^7 y_2^2, & \qquad y_2(0)=0\\
     y_3' &\!\!\!=\!\!\!&  & & 3\cdot 10^7 y_2^2, &\qquad  y_3(0)=0,
  \end{array}
\end{equation}
that,  rewritten in a vector form, read
\begin{equation}\label{eq:RobertsonLin}
  \frac{\mathrm{d}}{\mathrm{d}t}\left[
  \begin{array} {c}
   y_1\\ y_2\\ y_3
  \end{array} \right]=
  \left[\begin{array} {rcc}
    -0.04 & 10^4y_3 & 0  \\
    0.04 &- 3\cdot 10^7 y_2-10^4y_3 & 0 \\
    0 \ & 3\cdot 10^7 y_2 &  0 
  \end{array} \right]
  \left[
  \begin{array} {c}
    y_1\\ y_2\\ y_3
  \end{array} \right]\!, 
\end{equation}
where the matrix is graph Laplacian.
This example fits into the framework of Theorem~\ref{thm:LawMassAction:Laplacian}, which comes later.

Note that the system can also be written in many different ways, for example
\begin{equation}\label{eq:RobertsonLinNO}
  \frac{\mathrm{d}}{\mathrm{d}t}\left[
   \begin{array} {c}
     y_1\\ y_2\\ y_3
   \end{array} \right] =
  \left[\begin{array} {rrr}
     -0.04 & 0 \ \ & 10^4y_2 \\
     0.04 &- 3\cdot 10^7 y_2 &-10^4y_2 \\
     0 \ & 3\cdot 10^7 y_2 &  0 \ \
  \end{array} \right]
  \left[
  \begin{array} {c}
    y_1\\ y_2\\ y_3
   \end{array} \right]\!,
\end{equation}
where now the matrix is no longer a graph Laplacian. As we will see, it is crucial to write properly  the equations for the numerical solutions to preserve their qualitative properties.

\

\paragraph{Example 2: The SIR model.}

The {\em Susceptible--Infected--Recovered (SIR) model\/} describes the temporal epidemic evolution in terms of three variables for the population: 
 $S(t)$: (Susceptible),  $I(t)$ (Infected) and $R(t)$ (Recovered). It is usually asssumed that the total population does not change during the infection period. $S, I$ and $R$ denote the fractions with respect to the total population: $S(t) + I(t) + R(t) \equiv 1$. This model was proposed in  \cite{kermack27act}
\begin{eqnarray}
 &&S' = -R_0 S I \ ,  \nonumber  \\
 &&I' = R_0 S I  - I \ ,\label{eq:SIR} \\
 &&R' =  I \ , \nonumber
\end{eqnarray}
where $R_0>0$ is the basic reproduction number, and the system can be written in the form
\begin{equation}\label{eq:SIRLin}
\frac{d}{dt}\!\left[
\begin{array} {c}
 S\\
 I\\
 R
\end{array} \right] =
\left[\begin{array} {rcc}
 -R_0 I & 0 & 0  \\
  R_0 I &- 1 & 0 \\
  0 \ & 1 &  0 
\end{array} \right]
\left[
\begin{array} {c}
 S\\
 I\\
 R
\end{array} \right]\,
\end{equation}
{{}{which is like \eqref{eq:ODE2} with ${\bf y}=[S,I,R]^\top$ and}
the matrix is evidently a graph Laplacian.

\

\paragraph{Example 3: Laplacian dynamics on graphs (autonomous and linear).}
Graph Laplacian dynamics, ${\bf y}' = \mathcal{L}(G) {\bf y}$, where $\mathcal{L}(G)$ is a constant matrix, representing the Laplacian of a directed graph $G$, gives rise to a large class of applications in  biochemical kinetics, including Michaelis--Menten enzyme kinetics, allosteric enzymes, G-protein coupled receptors, ion channels, and gene regulation   \cite{gunawardena2012linear} (equation (3)).
Discussion of conditions under which such  linear systems always  converge to a steady state, and discussion of the sense in which that might be considered unique is given in \cite{mirzaev2013Laplacian}.
That linear setting ${\bf y}' = \mathcal{L}(G) {\bf y}$ is a special case of the more general framework here where we focus on the exact nonlinear model in \eqref{eq:ODE2}.

\

\paragraph{Example 4: Cardiac ion channels (nonautonomous and linear).}
Nonautonomous Laplacian systems, 
$
{\bf y}' = A(t) {\bf y},
$
 have many important applications, including cardiac ion channel kinetics \cite{earnshaw2010global,earnshaw2010invariant}.
In special cases, there are also exact solutions for the dynamical solutions, such as the explicit Magnus formul\ae{} in \cite{iserles2019applications}, and closely related invariant manifolds of binomial-like solutions.

\

\paragraph{Example 5: MAPK cascade (autonomous and nonlinear).}
The mitogen-activated protein kinase (MAPK) cascade is fundamental in cell signalling biology and cancer biology, and it is modelled by eighteen differential equations with rates  given by the Law of Mass Action, together with some linear conservation laws   \cite{qiao2007bistability}.
By our Theorem~\ref{thm:LawMassAction:Laplacian}, in the sequel, this MAPK model fits our framework of  \eqref{eq:ODE2}, subject to the remarks we make following Proposition~\ref{prop:GraphLaplacians}.
The Laplacian dynamics mentioned in the above constant coefficient and linear examples, where  convergence to a steady state is common \cite{mirzaev2013Laplacian}, makes it tempting to conjecture that
the model we focus on here in \eqref{eq:ODE2}, likewise always converges to a steady state.
However, a counterexample is provided by the MAPK cascade, which can be modelled by our nonlinear Laplacian dynamics \eqref{eq:ODE2}, and which has been shown by numerical simulations to exhibit both bistability and  oscillations  \cite{qiao2007bistability}.

We have taken the model of \cite{hadavc2017minimal} (Table 3, Fig 3, equations (12)--(17)), which is closely related to the MAPK cascade, and rewritten it here in the form of our model \eqref{eq:ODE2}, to show that it is clearly an example of the Laplacian dynamics  that we study in this paper:
\begin{equation} 
\label{eq:Oscillation:model}
  \frac{\mathrm{d}}{\mathrm{d}t}\!\left[
  \begin{array} {c}
   y_1\\ y_2\\ y_3 \\ y_4 \\ y_5 \\ y_6
  \end{array} \right]=
  \left[\begin{array} {cccccc}
    -k_7 - k_1y_2 & 0 & 0 &k_2 & 0&k_6 \\
      0& -k_1y_1 & k_5 &0 & 0& 0\\
     0& 0 &  -k_3y_1 - k_5 & k_2 & k_4 &0 \\
     (1- \alpha) k_1y_2& \alpha k_1y_1 & 0 & -k_2& 0&0 \\
     0& 0& k_3y_1 &0 &-k_4 & 0\\
     k_7 &  0&  0& 0& 0&-k_6 
  \end{array} \right]
  \left[
  \begin{array} {c}
   y_1\\ y_2\\ y_3 \\ y_4 \\ y_5 \\ y_6
  \end{array} \right]\!. 
\end{equation}
We take the same rate constants $k_1=\frac{100}{3}$, $k_2=\frac13$, $k_3=50$, $k_4=\frac12$, $k_5=\frac{10}{3}$, $k_6=\frac{1}{10}$, $k_7=\frac{7}{10}$, and  initial state $y(0) = [0.1,0.175,0.15,1.15,0.81,0.5]^\top$.
Note that we have $ {\bf y}' = A(\alpha,{\bf y}){\bf y}$, where $\alpha\in[0,1]$ is a parameter we can freely choose in this interval and the  matrix $A(\alpha,{\bf y})$ has the same pattern of signs as a Laplacian, but that a column of $A(\alpha,{\bf y})$ does not always sum to zero, so this fits our framework of  \eqref{eq:ODE2},  subject to the remarks we make following Proposition~\ref{prop:GraphLaplacians}, and this is also an example of our later Theorem~\ref{thm:LawMassAction:Laplacian}.
This model  possesses two conservation of mass laws, namely both
$y_1+y_4+y_6$ and $y_2+y_3+y_4+y_5$ are constants, which have physical interpretation in terms of the total enzyme of two types of kinases.
Those two conservation laws correspond to ${\bf w}_1 = [1,0,0,1,0,1]^\top$, and ${\bf w}_2=[0,1,1,1,1,0]^\top$, respectively.
Note that we have
\[
{\bf w}_1^\top A(\alpha,{\bf y}){\bf y} = 0 , \qquad {\bf w}_2^\top A(\alpha,{\bf y}) {\bf y} = 0, 
\]
and this is irrespective of the value of $\alpha\in(0,1)$. However, if we take   $\alpha=0$
we have that
\[
{\bf w}_1^\top A(0,{\bf y}) = 0 , \qquad {\bf w}_2^\top A(0,{\bf y}) \ne 0, 
\]
while for  $\alpha=1$ 
\[
{\bf w}_1^\top A(1,{\bf y}) \ne 0 , \qquad {\bf w}_2^\top A(1,{\bf y}) = 0. 
\]
It should be possible to use methods based on matrix exponentials  (such as the methods proposed in this paper) to respect e.g.\ the second conservation law, if we take  $\alpha=1$ because ${\bf w}_2^\top A({\bf y}) = 0$, so ${\bf w}_2^\top \exp(tA({\bf y})) = {\bf w}_2^\top$.
However, because ${\bf w}_1^\top A({\bf y}) \ne 0$, it will be difficult 
{(and probably impossible)} to  maintain exactly the first conservation law by methods that compute matrix exponentials.\footnote{
{The situation whereby it is impossible to satisfy several conservation laws under discretisation -- except, of course, by the exact solution -- is familiar in Geometric Numerical Integration \cite{hairer10gni}.}}
This is typical of applications in chemical kinetics, and for example, the famous Michaelis--Menten enzyme kinetics model (which always converges to a unique and  simple steady state) also fits the framework, with a matrix that has the same pattern of signs as a Laplacian, but that does not have zero column sum, and the model still has two simple well-known linear conservation laws.  
Significantly, by numerical simulation, it has been shown that solutions of this model \eqref{eq:Oscillation:model} show oscillations \cite{hadavc2017minimal} (Fig.~5).

\setcounter{equation}{0}
\section{Positivity preserving second-order methods}


Let us first consider the particular case in which the matrix $A$ is constant. Then the exact solution is given via the exponential:
\begin{displaymath}
 {\bf y}(t)=\e^{tA}{\bf y}_0.
\end{displaymath}
 If $A$  is a graph Laplacian matrix it is a consequence of Theorem 2 that $\sigma(\e^{tA})\subset \{z\in\BB{C}\,:\,|z|\leq 1\}$, hence the solution is stable (subject to the discussion of stability we gave earlier,  in the nonsymmetric case).

The exponential of a graph Laplacian matrix is fundamental to the work of this paper, and this calls for a more detailed study of its qualitative properties.

\subsection{The exponential of a graph Laplacian matrix}
{{}{We begin with column sums for an arbitrary square matrix.
\begin{proposition}
  \label{prop:column:sums}
    Suppose  that  ${\bf 1}^\top A={\bf 0}^\top$. Then ${\bf 1}^\top\e^{A}={\bf 1}^\top$.
\end{proposition}}
{{}{
\begin{proof}
By the series definition of the exponential 
\[
{\bf 1}^\top\e^{A}= {\bf 1}^\top \left(\sum_{n=0}^{\infty} \frac{A^n}{n!} \right)= ({\bf 1}^\top I) +  \sum_{n=1}^{\infty}({\bf 1}^\top A) \frac{A^{n-1}}{n!} =  {\bf 1}^\top + \sum_{n=1}^{\infty}({\bf 0}^\top)\frac{A^{n-1}}{n!}  ={\bf 1}^\top. 
\]
\end{proof}}

\noindent
{{}{{\bf Remark.}
Replace $A$ by $tA$ in the Proposition to see that, as a corollary,  if ${\bf 1}^\top A={\bf 0}^\top$, then ${\bf 1}^\top\e^{tA}={\bf 1}^\top$.}
~\\[6pt]
{{}{{\bf Remark.}
Graph Laplacians have the property ${\bf 1}^\top A={\bf 0}^\top$ by definition, so for graph Laplacians it is also true that ${\bf 1}^\top\e^{tA}={\bf 1}^\top$.}

We need the following elements of the {\em Perron--Frobenius theory\/} \cite{berman79nmm} (p.~26--27). Let $B\in\BB{R}^{d\times d}$, $B\succeq O$. Then $\rho(B)$ is an eigenvalue of $B$ and we can choose the corresponding eigenvector ${\bf v}$ such that ${\bf v}\succeq{\bf 0}$. Moreover, if in addition $B$ is {\em irreducible\/} then $\rho(B)$ is a simple eigenvalue and ${\bf v}$ is the only eigenvector of $B$ with nonnegative entries. 

Let $a^*=\min_{i=1,\ldots,d} A_{i,i}<0$ and set $\tilde{A}=A-a^*I$. Then
\begin{displaymath}
  \e^{tA}=\e^{ta^*I+t\tilde{A}}=\e^{ta^*}\e^{t\tilde{A}}.
\end{displaymath}
Since $\tilde{A}\succeq O$, all its nonnegative powers are also nonnegative and we deduce that $\e^{t\tilde{A}}\succeq O$. Therefore $\e^{tA}\succeq O$. 
Indeed,  the Mittag--Leffler matrix function of a graph Laplacian, $E_{\alpha}(A t^{\alpha}) $,  is likewise a stochastic matrix, i.e.  $E_{\alpha}(A t^{\alpha}) \succeq O$, and all entries are positive, and columns sum to unity \cite{macnamara2017fractional}.
 Here the Mittag--Leffler function $E_{\alpha}(z) = \sum_{k=0}^{\infty} z^{k}/\Gamma(\alpha k +1)$  is a one-parameter generalisation of the exponential, and the exponential is recovered as the special case once $\alpha \rightarrow 1$.
 Furthermore, when $A$ is Laplacian then the pattern of signs in the resolvent, and the properties of $M$-matrices, show that for large $n$, all entries of the matrix $\left( I - \frac{t}{n}A \right)^{-n}$ are nonnegative. 
This suggests the results we derive here may be extended to more general settings.

Additionally, once $A$ is irreducible and we denote by ${\bf w}\succeq{\bf 0}$ the left eigenvector of $A$ corresponding to the zero eigenvalue, then ${\bf w}^\top \tilde{A}=-a^*{\bf w}^\top$. Since $a^*<0$, we deduce that $|a^*|=\rho(\tilde{A})$ and ${\bf w}={\bf 1}$. 

\begin{proposition}
  \label{prop:Spectrum}
  \begin{displaymath}
    \sigma(\tilde{A})\subset\{z\in\BB{C}\,:\, |z|\leq |a^*|\}.
  \end{displaymath}
\end{proposition}

\begin{proof}
  By the Gerschgorin theorem applied to the columns of $\tilde{A}$ (or the standard Gerschgorin theorem applied to $\tilde{A}^\top$) and because $A_{j,i}\geq0$ for $i\neq j$,  we have
  \begin{displaymath}
    \sigma(\tilde{A})\subset \bigcup_{i=1}^d \left\{z\in\BB{C}\,:\, |z-A_{i,i}+a^*|\leq \sum_{\stackrel{\scriptstyle j=1}{j\neq i}}^d A_{j,i}\right\}
  \end{displaymath}
  and the proof follows. 
\end{proof}


\begin{proposition}
  \label{prop:ExpVec}
  Let $\BB{R}^d\ni{\bf p}\succeq{\bf 0}$ be such that ${\bf 1}^\top {\bf p}=1$ and set ${\bf q}=\e^{tA}{\bf p}$. Then for every $t\geq 0$ ${\bf q}\succeq{\bf 0}$ and ${\bf 1}^\top {\bf q}=1$.
\end{proposition}

\begin{proof}
  We deduce at once that ${\bf q}\succeq{\bf 0}$ because $\e^{tA}\succ O$. Moreover, ${\bf 1}^\top{\bf q}={\bf 1}^\top \e^{tA}{\bf p}={\bf 1}^\top{\bf p}=1$, concluding the proof.
\end{proof}

\paragraph{Remark.} Note that all previously stated results apply to maps of the form ${\bf z}= \e^{\sigma S}{\bf x}$ where ${\bf x}\succeq {\bf 0}$, $S$ is a graph Laplacian and $\sigma$ is a non-negative constant. Since $S$ can have large negative eigenvalues, taking negative values of $\sigma$ is likely to lead to a poorly conditioned problem where negative solutions can occur and this compels us to avoid this choice. In the sequel we propose several methods that involve maps of the form  $\e^{\sigma S}$ with $S$ being graph Laplacian, and we will see that condition  $\sigma>0$   limits the order of the methods to two in the time step, an order barrier.

{{}{
Since $\e^{tA}=\left[ I - tA \right]^{-1} + {\cal O}(h^2)$, the following well known result will be useful in the sequel.
\begin{proposition}
  \label{prop:InvPositivity}
	If $A$ is graph-Laplacian then  $\left[ I - tA \right]^{-1}\succeq O$.
\end{proposition}
\begin{proof}
Given $B=I - tA $ we have that $B_{ii}>0, \forall i$ and $B_{i,j}\leq 0, \ i\neq j$. Since $\sigma(A)\setminus\{0\}\in\BB{C}^-$ then we have $\sigma(I - tA)\in\BB{C}^+$, which is an $M$-matrix whose inverse has only non negative elements. 
\end{proof}
}

\subsection{Splitting methods}


Splitting methods are frequently used to solve differential equations that are separable into solvable parts. However, for stiff as well as for non separable problems it is more convenient to proceed as follows \cite{blanes19otc}.
Let us  consider the following system in the extended space
\begin{eqnarray*}
 {\bf x}' &\!\!\!=\!\!\!& A({\bf z}){\bf x}, \qquad {\bf x}(0)={\bf x}_0={\bf y}_0,\\
 {\bf z}' &\!\!\!=\!\!\!& A({\bf x}){\bf z}, \qquad {\bf z}(0)={\bf z}_0={\bf y}_0,
\end{eqnarray*}
where  ${\bf x}(t)={\bf y}(t)={\bf z}(t)$. The system is separable into two solvable parts
\begin{displaymath}
  {\cal A}:\quad\left\{
  \begin{array}{rcl}
   {\bf x}' &\!\!\!=\!\!\!& A({\bf z}){\bf x}, \\[2pt]
   {\bf z}' &\!\!\!=\!\!\!& 0
  \end{array}
  \right.
   \qquad \Rightarrow \qquad 
  \left\{
  \begin{array}{rcl}
    {\bf x}(t) &\!\!\!=\!\!\!& \e^{t A({\bf z}_0)}{\bf x}_0,\\[2pt]
    {\bf z}(t) &\!\!\!=\!\!\!& {\bf z}_0
   \end{array}
   \right.
\end{displaymath}
and
\begin{displaymath}
  {\cal B}:\quad\left\{
  \begin{array}{rcl}
   {\bf x}' &\!\!\!=\!\!\!& 0 , \\[2pt]
   {\bf z}' &\!\!\!=\!\!\!& A({\bf x}){\bf z}
  \end{array}
  \right.
  \qquad \Rightarrow \qquad 
  \left\{
  \begin{array}{rcl}
   {\bf x}(t) &\!\!\!=\!\!\!& {\bf x}_0,\\[2pt]
   {\bf z}(t) &\!\!\!=\!\!\!& \e^{t A({\bf x}_0)}{\bf z}_0.
  \end{array}
  \right.
\end{displaymath}

We solve the system with the symmetric second order Strang splitting method,
i.e.\  advance half a step with ${\cal A}$ followed by a step with ${\cal B}$ and  conclude with another half a step with ${\cal A}$:
\begin{eqnarray}
 {\bf x}_{1/2} &\!\!\!=\!\!\!& \e^{\frac12 hA({\bf z}_0)}{\bf x}_0, \nonumber \\
 {\bf z}_1 &\!\!\!=\!\!\!& \e^{hA({\bf x}_{1/2})}{\bf z}_0,   \label{eq:SplittingMethod} \\
 {\bf x}_{1} &\!\!\!=\!\!\!& \e^{\frac12 hA({\bf z}_1)}{\bf x}_{1/2} .  \nonumber 
\end{eqnarray}

Since ${\bf z}_0\succeq 0$,  the frozen matrix $A({\bf z}_0)$ is a graph Laplacian, therefore ${\bf x}_{1/2}\succeq 0$ and preserves the 1-norm, and similarly for ${\bf z}_1$ and ${\bf x}_1$.

In addition,  ${\bf x}_1$ and ${\bf z}_1$ correspond to symmetric second order approximations: ${\bf z}_1$ can be seen as the exponential midpoint and ${\bf x}_1$ as the exponential trapezoidal rule. 
Then, we can advance the solution either with ${\bf z}_1$ or with ${\bf x}_1$ but, in general,  more accurate results are obtained with the smoothing technique, i.e.\ taking the solution for the next step as the average
\begin{equation} \label{eq:smoothing} 
  {\bf y}_1=\frac12({\bf x}_1+{\bf z}_1)
\end{equation}
where again the 1-norm is preserved and all components of ${\bf y}_1$ are nonnegative. The Lie group structure is not preserved by this linear combination, but this is not a property that concerns us in the present context. 
In addition, the difference ${\bf x}_1-{\bf z}_1$ can be taken as an estimate of local error, using the scheme as a variable time-step algorithm in order to get more accurate results.

\paragraph{Remark.} If $A$ is graph Laplacian and irreducible then \eqref{eq:SplittingMethod} is a time-symmetric second order method that preserves mass and positivity unconditionally (the average \eqref{eq:smoothing} breaks  time symmetry) and converges to the steady state solution.
{{}{Let $\yy_f$ be the steady state solution, then ${\bf f}(\yy_f)=A(\yy_f)\yy_f={\bf 0}$. Since $\sigma(A)\setminus\{0\}\in\BB{C}^-$ where  0 is a simple eigenvalue and the method is a composition of exponentials of $A$, it must converge to a steady state solution, say $\hat\yy_f$. However, we observe that if we take $\yy_0=\yy_f$ then
it is trivial to check that ${\bf x}_{1/2}={\bf z}_1={\bf x}_{1}=\yy_f$, so $\yy_1=\yy_f$ and then $\hat\yy_f=\yy_f$.
}

{{}{
Note also that
\[
 {\bf u} = \left[ I - \frac12 hA({\bf z}_0) \right]^{-1}{\bf x}_0 =  \e^{\frac12 hA({\bf z}_0)}{\bf x}_0  + {\cal O}(h^2)
\]
which is a first order approximation to the exponential that, by  our earlier  Proposition~\ref{prop:InvPositivity}, still preserves positivity. We can replace  $A({\bf x}_{1/2})$ by $A({\bf u})$ 
in \eqref{eq:SplittingMethod} and, since this matrix is multiplied by $h$, the method retains  second order of accuracy. 
}

\

\paragraph{The non-autonomous case.}
Let us now consider the non-autonomous system
\begin{displaymath}
  {\bf y}' = A(t,{\bf y}) {\bf y}, \qquad {\bf y}(0)={\bf y}_0.
\end{displaymath} 
This occurs, for example, when a chemical reaction  takes place at variable temperature and the coefficients $k_i(t)$ are time dependent or when the parameter $R_0$ in the SIR model changes due to political decisions, variations in behaviour or the evolution of a pathogen.

In this case we duplicate the system, but taking the time as two dependent variables
\begin{displaymath}
\begin{array}{rcll}
 {\bf x}' &\!\!\!=\!\!\!& A(z_t,{\bf z}){\bf x}, \qquad &{\bf x}(0)={\bf x}_0={\bf y}_0\\
 x_t' &\!\!\!=\!\!\!& 1, \qquad &x_t(0)=t_0\\[2pt]
 {\bf z}' &\!\!\!=\!\!\!& A(x_t,{\bf x}){\bf z}, \qquad &{\bf z}(0)={\bf z}_0={\bf y}_0\\
 z_t' &\!\!\!=\!\!\!& 1, \qquad &z_t(0)=t_0
\end{array}
\end{displaymath}
where  ${\bf x}(t)={\bf y}(t)={\bf z}(t)$ and $x_t(t)=z_t(t)=t$: the system is now autonomous and separable into solvable parts: the outcome is an algorithm similar to \eqref{eq:SplittingMethod},
\begin{eqnarray*}
 {\bf x}_{1/2} &\!\!\!=\!\!\!& \e^{\frac{h}2A(t_0,{\bf z}_0)}{\bf x}_0, \\
 {\bf z}_1 &\!\!\!=\!\!\!& \e^{hA(t_0+h/2,{\bf x}_{1/2})}{\bf z}_0,   \\
 {\bf x}_{1} &\!\!\!=\!\!\!& \e^{\frac{h}2A(t_0+h,{\bf z}_1)}{\bf x}_{1/2} ,  
\end{eqnarray*}
and finally
\begin{equation} \label{eq:smoothing:nonaut} 
  {\bf y}_1=\frac12({\bf x}_1+{\bf z}_1).
\end{equation}

\subsection{Magnus integrators}

A more general procedure to construct higher-order methods is to consider Magnus integrators.

Let $A:\BB{R}\times \BB{R}^d\rightarrow\BB{R}^{d\times d}$. We consider the equation
\begin{equation}
  \label{eq:GrLapODE}
  \yy '=A(t,\yy )\yy ,\quad t\geq 0,\qquad\yy (0)= \yy _0,
\end{equation}
and suppose that $\yy _n\simeq\yy (t_n)$. One approach toward the solution of \eqref{eq:GrLapODE} is
\begin{eqnarray}
  \nonumber
  \yy ^{[0]}  &\equiv&\yy _n,\\
  \nonumber
  {\yy ^{[m+1]}}'& = &A(t,\yy ^{[m]}(t))\yy ^{[m+1]},\quad \yy ^{[m+1]}(t_n)=\yy _n,\qquad m=0,1,\ldots,m^*-1,\\
  \yy _{n+1}& = &\yy ^{[m^*]\,}\!(t_{n+1}),\qquad \mbox{where}\qquad t_{n+1}=t_n+h_n,
  \label{eq:Iterate}
\end{eqnarray}
that corresponds to an approximation to the exact solution to order $m^*$.
The linear ODE in \eqref{eq:Iterate} can be solved e.g.\ by Magnus series expansion \cite{magnus54ote}
 (see also \cite{blanes09tme,iserles00lgm,iserles99ots}
 and references therein). 
{{}{For simplicity, we first consider the autonomous case, with $A(\yy)$,  and next we show the results for the non-autonomous problem.}

For example, for $m^*=1$ we have ${\yy ^{[1]}}'=A(\yy _n)\yy ^{[1]}$, therefore $\yy ^{[1]}(t)=\e^{(t-t_n)A(\yy _n)}\yy _n$  and 
we obtain the first-order method
\begin{equation}
  \label{eq:ExpoRep}
  \yy _{n+1}=\e^{h_nA(\yy _n)}\yy _n.
\end{equation}
Letting $m^*=2$ leads to a second-order method ${\yy ^{[2]}}'=A(\e^{(t-t_n)A(\yy _n)}\yy _n)\yy ^{[2]}$, whose Magnus solution truncated to the first term that provides second order approximations in the time step is
\begin{eqnarray*}
  \yy ^{[2]}(t)&\!\!\!=\!\!\!&\exp\!\left(\int_{t_n}^t A(\e^{(\tau-t_n)A(\yy _n)}\yy _n)\D\tau\right)\!\yy _n\\
  &\!\!\!\approx\!\!\!&\exp\!\left(\frac{t-t_n}{2}\left(A(\yy _n)+A(\e^{(t-t_n)A(\yy _n)}\yy _n)\right)\!\!\right)\!\yy _n
\end{eqnarray*}
(note that the approximation of the integral with the trapezoidal rule is fully consistent with second order). 
 This results in the second-order method
\begin{equation}
  \label{eq:ExpoRepa}
  \yy _{n+1}=\exp\!\left(\frac{h_n}{2}\left(A(\yy _n)+A(\e^{h_nA(\yy _n)}\!\yy _n)\right)\!\right)\!\yy _n.
\end{equation}

If we consider instead the midpoint rule we have
\begin{equation}
  \label{eq:ExpoRepb}
  \yy _{n+1}= \exp\!\left(h_nA(\e^{\frac12 h_nA(\yy _n)}\yy _n)\right)\yy _n.
\end{equation}
Note that \eqref{eq:ExpoRepb} coincides with ${\bf z}_1$ in \eqref{eq:SplittingMethod} for the first step. This method requires only two exponentials but it is not time symmetric. If it is important to preserve time symmetry,  the three-exponential method  \eqref{eq:SplittingMethod} should be used, otherwise this simple and cheaper scheme suffices.

\paragraph{Remark.} If $A$ is graph Laplacian and irreducible then the first-order methods \eqref{eq:ExpoRep} as well as the second order methods \eqref{eq:ExpoRepa} and \eqref{eq:ExpoRepb} preserve mass and positivity unconditionally and converge to the steady state solution 
{{}{similarly to the previous splitting methods}.


\

We can easily apply these Magnus integrators to  non-autonomous problems.
The first-order method is, obviously 
\begin{equation}
  \label{eq:1.3t}
  \yy _{n+1}=\e^{h_nA(t_n,\yy _n)}\yy _n.
\end{equation}
The trapezoidal second-order method is given by
\begin{equation}
  \label{eq:1.3at}
  \yy _{n+1}=\exp\!\left(\frac{h_n}{2}\left(A(t_n,\yy _n)+A(t_{n+1},\e^{h_nA(t_n,\yy _n)}\!\yy _n)\right)\!\right)\!\yy _n,
\end{equation}
while the corresponding second-order midpoint rule method is
\begin{equation}
  \label{eq:1.3bt}
  \yy _{n+1}= \exp\!\left(h_nA(t_n+\frac{h_n}{2},\e^{\frac12 h_nA(t_n\yy _n)}\yy _n)\right)\yy _n.
\end{equation}

{{}{Note that if we consider the first order approximation to $\e^{h_nA(t_n,\yy _n)}$ or $\e^{\frac12 h_nA(t_n\yy _n)}$ given by
\[
  {\bf u}_1=[I-h_nA(t_n,\yy _n)]^{-1} \yy_n
	  \qquad \mbox{or} \qquad
  {\bf u}_2=[I-\frac{h_n}{2}A(t_n,\yy _n)]^{-1} \yy_n
\]
as the internal stages in \eqref{eq:1.3at} or \eqref{eq:1.3bt} then the new and cheaper schemes read
\begin{equation}
  \label{eq:1.3atN}
  \yy _{n+1}=\exp\!\left(\frac{h_n}{2}\left(A(t_n,\yy _n)+A(t_{n+1},{\bf u}_1)\right)\!\right)\!\yy _n,
\end{equation}
or
\begin{equation}
  \label{eq:1.3btN}
  \yy _{n+1}= \exp\!\left(h_nA(t_n+\frac{h_n}{2},{\bf u}_2)\right)\yy _n.
\end{equation}
and still preserve mass and positivity as well as the second order accuracy. We can either compute the exponentials to high accuracy, or to look for cheaper approximations that preserve both mass and positivity, this being  a problem to be studied further in the future.
}

\subsection{Patankar methods}

A well-known approach toward preservation of mass and positivity are  Runge--Kutta--Patankar methods \cite{burchard03hoc,kopecz18ooc,kopecz2018upa,offner2020aho,patankar80nht}. The idea is to use Runge--Kutta-like methods for {\em pro\-duc\-tion--destruction systems\/} in chemical kinetics, of the form 
\begin{equation}\label{eq:ProdDestruc}
  y_k'=\sum_{j=1}^d p_{k,j}(t,{\bf y})-\sum_{j=1}^d d_{k,j}(t,{\bf y}),\qquad k=1,\ldots,d,
\end{equation}
where $p_{k,j}(t,{\bf y}),d_{k,j}(t,{\bf y})\geq0$. 
The first order Patankar method is given by
\[
 y_{n+1,k}=y_{n,k}+h\sum_{j=1}^d \left[p_{k,j}(t_n,{\bf y}_n)-d_{k,j}(t_n,{\bf y}_n)\frac{y_{n+1,k}}{y_{n,k}}\right]\!,\qquad k=1,\ldots,d.
\]
This method preserves positivity but does not preserve mass. 
In \cite{burchard03hoc} the authors propose a Modified Patankar Euler scheme (MPE) given by
\begin{equation}\label{eq:Patankar1}
 y_{n+1,k}=y_{n,k}+h\sum_{j=1}^d \left[p_{k,j}(t_n,{\bf y}_n)\frac{y_{n+1,j}}{y_{n,j}}-d_{k,j}(t_n,{\bf y}_n)\frac{y_{n+1,k}}{y_{n,k}}\right]\!,\qquad k=1,\ldots,d,
\end{equation}
which preserves both mass and positivity unconditionally. 
Note that we can write \eqref{eq:ProdDestruc} in our graph-Laplacian notation as
\begin{equation*}\label{}
  {\bf y}' = {\bf f}({\bf y})=  A(t,{\bf y}) {\bf y}
\end{equation*}
with
\[ 
  A_{k,j}(t,{\bf y}) = p_{k,j}(t,{\bf y})\frac1{y_{j}}-d_{k,j}(t,{\bf y})\frac{\delta_{k,j}}{y_{k}}
\]
where $\delta_{k,j}=0$ if $k\neq j$ and  $\delta_{k,k}=1$. 
Then, \eqref{eq:Patankar1} can be written in matrix form as 
\begin{equation}\label{modPatankarEuler}
{\bf y}_{n+1}={\bf y}_n + hA(t_n,{\bf y}_n) {\bf y}_{n+1}
\end{equation}
that preserves mass (because $A(t_n,{\bf y}_n)$ is graph Laplacian) and as already shown preserves positivity.
Note that
\[
{\bf y}_{n+1}=[I - hA(t_n,{\bf y}_n)]^{-1} {\bf y}_{n}= e^{hA(t_n,{\bf y}_n)}{\bf y}_n
+ {\cal O}(h^2),
\]
corresponding to a first order rational approximation to the first order exponential Magnus integrator. 


The second-order Modified Patankar--Runge--Kutta scheme (MPRK) is given by 
\begin{eqnarray*}
  u_k          &=&  y_{n,k}+h\sum_{j=1}^d \left[p_{k,j}(t_n,{\bf y}_n)\frac{u_j}{y_{n,j}}-d_{k,j}(t_n,{\bf y}_n)\frac{u_k}{y_{n,k}}\right]\!,\qquad k=1,\ldots,d,\\
  y_{n+1,k} &=&  y_{n,k}+\frac{h}{2}\sum_{j=1}^d    \{  [p_{k,j}(t_n,{\bf y}_n)+p_{k,j}(t_{n+1},{\bf u})]\frac{y_{n+1,j}}{u_j} -  \\
                  & &  \;\;\;\;\;\;\;\;\;\;\;\;\;\;\;\;\;\;\;\;\; [d_{k,j}(t_n,{\bf y}_n)+d_{k,j}(t_{n+1},{\bf u})] \frac{y_{n+1,k}}{u_k} \}
\end{eqnarray*}
\cite{burchard03hoc} (eq.~27), which can be written in matrix form as 
\begin{eqnarray}
{\bf u} &\!\!\!=\!\!\!& [ I - hA(t_n,{\bf y}_n)]^{-1} {\bf y}_{n} \nonumber  \\
{\bf y}_{n+1} &\!\!\!=\!\!\!& \left[ I - \frac{h}{2}\Big(A(t_n,{\bf y}_n)D({\bf y}_n,{\bf u}) +  A(t_{n+1},{\bf u}) \Big) \right]^{-1}   {\bf y}_{n} 
\end{eqnarray}
where
\[
  D({\bf y}_n,{\bf u}) = {\rm diag}\left(\frac{y_{n,1}}{u_1},\ldots,\frac{y_{n,d}}{u_d}\right)
\]
and ${\bf y}_n=(y_{n,1},\ldots,y_{n,d})^\top$, ${\bf u}=({u_1},\ldots,u_{d})^\top$.

{{}{Note that ${\bf u}$ coincides with  ${\bf u}_1$ in \eqref{eq:1.3at} and then the modified Patankar method can be considered as a particular second order approximation to the second order Magnus method \eqref{eq:1.3at}. Obviously, different second order approximations to this exponential or to the method using the midpoint rule \eqref{eq:1.3bt} would lead to different modified second order Patankar methods.
}

Note that during the integration some of the values $u_i$ may approach zero. 
In this case one may take, for example, $\frac{y_{n,i}}{u_i}=0$ when $u_i$ is smaller than a given tolerance. Some caution is required if any component of the solution is very close to zero and suddenly grows, as it happens with some of the numerical examples we will consider.


This method has shown a good performance on stiff problems \cite{burchard2005aom}. Higher order modified Patankar methods have also been recently obtained in the literature \cite{formaggia2011pac,kopecz2018upa,offner2020aho} and it would be interesting to find if there is any connection with our exponential integrators.

{{}{There are other families of methods which consider some kind of adaptive time steps which depend on the phase space and the time step which allows to preserve positivity as well as the linear invariants \cite{avila2020act,broekhuizen2008aia,martiradonna2020ggc}
 but they are not considered in this work and a proper study of their performance with respect to the new methods is left for  future research.}



\subsection{Higher order methods}

Continuing in this vain, 
\begin{displaymath}
  {\yy^{[3]}}'=A\!\left(\exp\!\left(\int_{t_n}^t A(\e^{(\tau-t_n)A(\yy _n)}\yy _n)\D\tau\!\right)\!\yy _n\!\right)\!\yy ^{[3]}=A_2(t)\yy ^{[3]}
\end{displaymath}
and a fourth-order Magnus reads (this is a 4th-order approximation to $\yy^{[3]}$ which is a third order approximation to the exact solution, so the methods will be of order three)
\begin{displaymath}
  \yy _{n+1}=\exp\!\left(\int_{t_n}^{t_{n+1}} A_2(\tau)\D\tau -\frac12 \int_{t_n}^{t_{n+1}} \int_{t_n}^\tau [A_2(\tau),A_2(\eta)]\D\eta\D\tau\right)\!\yy _n.
\end{displaymath}
The temptation is now to discretise using standard Magnus quadrature at Gauss--Legendre points but this does not work because the definition of $A_2$ itself contains an integral. Moreover,  the critical issue is the dependence of $A_2$ on $t$, not on $\yy _n$. 

We approximate
\begin{displaymath}
  \yy _{n+1}\approx \exp\!\left(\frac{h_n}{2} (\mathcal{B}_1+\mathcal{B}_2)+\frac{\sqrt{3}}{12} h_n^2 [\mathcal{B}_1,\mathcal{B}_2]\right)\!\yy _n,
\end{displaymath}
where
\begin{displaymath}
  \mathcal{B}_1=A_2(t_n+(\Frac12-\Frac{\sqrt{3}}{6})h),\qquad \mathcal{B}_2=A_2(t_n+(\Frac12+\Frac{\sqrt{3}}{6})h)
\end{displaymath}
-- except that $A_2$ itself has a built-in integral,
\begin{displaymath}
  A_2(t)=A\!\left(\exp\!\left(\int_{t_n}^t A_1(\eta)d\eta\right)\!\yy _n\right)\!.
\end{displaymath} 
The simplest solution is to approximate that integral also by two-point Gauss--Legendre (note that the interval of integration in the inner integral is of length $(\frac12-\frac{\sqrt{3}}{6})h_n$ and we need to adjust quadrature points), whereby
\begin{eqnarray*}
  \mathcal{B}_1&\!\!\!\approx\!\!\!& A\!\left(\exp\!\left((\Frac14-\Frac{\sqrt{3}}{12})h_n\! \left[A_1(t_n+(\Frac13-\Frac{\sqrt{3}}{6} )h_n) + A_1(t_n+\Frac16 h_n)\right]\!\right)\!\yy _n\right)\!,\\
  \mathcal{B}_2&\!\!\!\approx\!\!\!& A\!\left(\exp\!\left((\Frac14+\Frac{\sqrt{3}}{12})h_n\! \left[ A_1(t_n +\Frac16 h_n)+A_1(t_n+(\Frac13+\Frac{\sqrt{3}}{6})h_n)\right]\!\right)\yy _n\!\right)\!.
\end{eqnarray*}
Brief explanation: the first integral is in the interval $[t_n,t_n+(\frac12-\frac{\sqrt{3}}{6})]$ and the Gauss--Legendre nodes $\frac12\pm\frac{\sqrt{3}}{6}$ need be multiplied by the length of the interval. Ditto in the second interval, $[t_n,t_n+(\frac12+\frac{\sqrt{3}}{6})]$ and we are saved a single function evaluation because, by happy coincidence, real numbers commute and $(\frac12-\frac{\sqrt{3}}{6})(\frac12+\frac{\sqrt{3}}{6})=(\frac12+\frac{\sqrt{3}}{6})(\frac12-\frac{\sqrt{3}}{6})=\frac16$.

Thus, altogether we need three function evaluations, one more than standard Magnus. Note moreover that the integration in $A_1$ is explicit,
\begin{displaymath}
  A_1(t)=A(\e^{(t-t_n)A(\yy _n)}).
\end{displaymath}
We observe that $\mathcal{B}_i, \ i=1,2$, are graph Laplacians, but this need not be the case for their commutator $[\mathcal{B}_1,\mathcal{B}_2]$. This problem can be bypassed using commutator-free Magnus integrators \cite{alvermann11hoc,blanes17hoc}.

~\\

\paragraph{Commutator-free Magnus integrators} We describe briefly, using an example, the construction of commutation-free integrators, based upon the work of \cite{blanes17hoc}. We approximate the solution across a single time step by
\begin{displaymath}
  \yy _{n+1}\approx 
	\exp\!\left(\frac{h_n}{2} (\beta \mathcal{B}_2+\alpha \mathcal{B}_1)\right)
	\exp\!\left(\frac{h_n}{2} (\alpha \mathcal{B}_2+\beta \mathcal{B}_1)\right)
	\!\yy _n,
\end{displaymath}
where
\begin{displaymath}
 \alpha = \frac12+\frac{\sqrt{3}}{3}, \qquad
 \beta = \frac12-\frac{\sqrt{3}}{3}
\end{displaymath}
and the algorithm is given by
\begin{eqnarray}
  {\bf x}_1 &\!\!\!=\!\!\!& \exp\!\left((\Frac13-\Frac{\sqrt{3}}{6} )h_nA(\yy _n)\right)\!\yy _n, \hspace*{50pt}  A_{1,1}=A({\bf x}_1),  \nonumber  	\\
  {\bf x}_2 &\!\!\!=\!\!\!& \exp\!\left(\Frac1{6} h_nA(\yy _n)\right)\!\yy _n, 	\hspace*{85pt}  A_{1,2}=A({\bf x}_2),  \nonumber  	\\
  {\bf x}_3 &\!\!\!=\!\!\! & \exp\!\left(\Frac13+\Frac{\sqrt{3}}{6} )h_nA(\yy _n)\right)\!\yy _n,	\hspace*{53pt}  A_{1,3}=A({\bf x}_3),  \nonumber  	\\
  {\bf x}_4 &\!\!\!=\!\!\!& \exp\!\left((\Frac14-\Frac{\sqrt{3}}{12} )h_n(A_{1,1}+A_{1,2})\right)\!\yy _n, \hspace*{20pt}  \mathcal{B}_1=A({\bf x}_4),  \nonumber  	\\
  {\bf x}_5 &\!\!\!=\!\!\!& \exp\!\left((\Frac14+\Frac{\sqrt{3}}{12} )h_n(A_{1,2}+A_{1,3})\right)\!\yy _n, \hspace*{20pt}  \mathcal{B}_2=A({\bf x}_5),  \nonumber  	\\
  {\bf x}_6 &\!\!\!=\!\!\!& \exp\!\left(\Frac12 h_n (\alpha \mathcal{B}_2+\beta \mathcal{B}_1)\right)\!\yy _n,  \nonumber  \\
  \yy _{n+1} &\!\!\!=\!\!\! & 	\exp\!\left(\Frac12 h_n (\beta \mathcal{B}_2+\alpha \mathcal{B}_1)\right)\!{\bf x}_6. \label{eq:CF3}
\end{eqnarray}

This is a seven-exponential method that might be useful when highly accurate results are desired and the cost of each exponential is not excessive. It preserves positivity for moderately stiff problems since it is conditionally positivity preserving. 
{{}{
If $\yy_n\succeq {\bf 0}$ then it is easy to see that  ${\bf x}_i\succeq {\bf 0}, \ i=1,2,3,4,5$ since $A_{1,1},A_{1,2},A_{1,3},\mathcal{B}_1,\mathcal{B}_2\in {\cal L}_d$. However,  positivity is guaranteed}
as long as the matrices
\begin{displaymath}
   \alpha \mathcal{B}_2+\beta \mathcal{B}_1, \qquad
	\beta \mathcal{B}_2+\alpha \mathcal{B}_1
\end{displaymath}
are graph Laplacians \cite{macnamara20sob}. Unfortunately,  $\beta<0$ but $\alpha/|\beta|=7+4\sqrt{3}\simeq 14$, and unless $\mathcal{B}_1, \mathcal{B}_2$  drastically change in a short time interval or their sparsity structure is `unlucky', their linear combinations  are likely to inherit graph-Laplacian structure. In other words, while preservation of graph Laplacians for this third-order method is not assured, it is highly likely in practice.

Finally, we present this Magnus integrator to be used on non-autonomous problems.
A third-order commutator-free method can be obtained following the same approximations as previously and taking, for example, the midpoint rule when approximating the intermediate integrals that  ensure the third order of accuracy for the method, resulting in the following algorithm
\[
A_1=A( t_n+(\Frac16-\Frac{\sqrt{3}}{12}) h_n,\yy _n), \quad\! 
		 A_2=A( t_n+\Frac1{12} h_n,\yy _n)  , \quad\!
		 A_3=A( t_n-(\Frac16-\Frac{\sqrt{3}}{12}) h_n,\yy _n),
\]
\begin{eqnarray}
  {\bf x}_1 &\!\!\!=\!\!\!& \exp\!\left((\Frac13-\Frac{\sqrt{3}}{6} )h_nA_1\right)\yy _n,  \hspace*{50pt}   A_{1,1}=A( t_n+(\Frac13-\Frac{\sqrt{3}}{6}) h_n,{\bf x}_1)  \nonumber  	\\
  {\bf x}_2 &\!\!\!=\!\!\!& \exp\!\left(\Frac1{6} h_nA_2\right)\yy _n, 	\hspace*{85pt}      A_{1,2}=A( t_n+\Frac1{6} h_n,{\bf x}_2)   \nonumber  	\\
  {\bf x}_3 &\!\!\!=\!\!\! & \exp\!\left((\Frac13+\Frac{\sqrt{3}}{6} )h_nA_3\right)\yy _n,	\hspace*{53pt}  A_{1,3}=A(t_n+(\Frac13+\Frac{\sqrt{3}}{6} )h_n,{\bf x}_3)  \nonumber  	\\
  {\bf x}_4 &\!\!\!=\!\!\!& \exp\!\left((\Frac14-\Frac{\sqrt{3}}{12} )h_n(A_{1,1}+A_{1,2})\right)\yy _n, \hspace*{20pt}  \mathcal{B}_1=A(t_n+(\Frac12-\Frac{\sqrt{3}}{6} )h_n,{\bf x}_4)  \nonumber  	\\
  {\bf x}_5 &\!\!\!=\!\!\!& \exp\!\left((\Frac14+\Frac{\sqrt{3}}{12} )h_n(A_{1,2}+A_{1,3})\right)\yy _n, \hspace*{20pt}  \mathcal{B}_2=A(t_n+(\Frac12+\Frac{\sqrt{3}}{6} )h_n,{\bf x}_5)  \nonumber  	\\
  {\bf x}_6 &\!\!\!=\!\!\!& \exp\!\left(\Frac12 h_n (\alpha \mathcal{B}_2+\beta \mathcal{B}_1)\right)\yy _n,  \nonumber  \\
  \yy _{n+1} &\!\!\!=\!\!\! & 	\exp\!\left(\Frac12 h_n (\beta \mathcal{B}_2+\alpha \mathcal{B}_1)\right){\bf x}_6. \label{eq:CF3t}
\end{eqnarray}

\setcounter{equation}{0}
\setcounter{figure}{0}
\section{Hidden graph Laplacian structures for polynomial ODEs}

\subsection{The recovery of graph Laplacian structure}

Given an ODE system of the form
\begin{equation}
  \label{eq:PolyODE}
  y_k'=\sum_{\ell=1}^d b^\ell_k y_\ell +\sum_{\ell=1}^d \sum_{i=1}^d a^{\ell,i}_k y_\ell y_i,\qquad k=1,\ldots,d,
\end{equation}
with suitable  initial conditions ${\bf y}(0)\succeq{\bf 0}$, ${\bf 1}^\top{\bf y}(0)=1$, we seek conditions so that it can be written in the form \eqref{eq:ODE2}, where the matrix $A({\bf y})$ is a graph Laplacian, namely that for every nonnegative ${\bf y}$ such that ${\bf 1}^\top{\bf y}=1$ it is true that $A_{k,k}({\bf y})\leq0$ and $A_{k,\ell}({\bf y)}\geq0$, $\ell\neq k$. Moreover, we seek constructive means of deriving such a matrix $A$, given \eqref{eq:PolyODE}.

Our first observation is that the representation of \eqref{eq:PolyODE} in the form \eqref{eq:ODE2} is additive, in the sense that if we can do so for two different right-hand sides of \eqref{eq:PolyODE}, we can do so for their sum. By the same token, if we can do so separately for the first sum and the second, double sum in \eqref{eq:PolyODE}, all we need is simply add the two representations. The first sum is trivial and corresponds to the constant-matrix representation ${\bf y}'=B{\bf y}$, where $B=(b^\ell_k)$ is a graph Laplacian. Consequently, the task at hand reduces to the derivation of a representation \eqref{eq:ODE2} of the system
\begin{displaymath}
  y_k'=\sum_{\ell=1}^d \sum_{i=1}^d a^{\ell,i}_k y_\ell y_i,\qquad k=1,\ldots,d.
\end{displaymath}

With greater generality, we may just as well consider the multinomial ODE system
\begin{displaymath}
  y_k'=\sum_{j=1}^m \sum_{\stackrel{\scriptstyle \ell_1+\cdots+\ell_d=j}{\ell_1,\ldots,\ell_d\geq0}} a_k^{\ell_1,,\ldots,\ell_d} y_1^{\ell_1}\cdots y_d^{\ell_d} =\sum_{j=1}^m \sum_{|{\mathbf{\ell}}|=j} a_k^{\mathbf{\ell}} {\bf y}^{\mathbf{\ell}},\qquad k=1,\ldots,d,
\end{displaymath}
with initial conditions ${\bf y}(0)={\bf y}_0\succeq{\bf 0}$, ${\bf 1}^\top{\bf y}_0=1$. Again, the challenge is to write it in the form \eqref{eq:ODE2} with a graph Laplacian $A({\bf y})$ and, again, we can use the same argument to split the task at hand into a sum of homogeneous problems of the form
\begin{equation}
  \label{eq:MultiODE2}
  y_k'=\sum_{|\mathbf{\ell}|=j} a^{\mathbf{\ell}}_k {\bf y}^{\mathbf{\ell}},\qquad k=1,\ldots,d
\end{equation}
for $j=2,\ldots,m$ -- the case $j=1$ is trivial. 

The problem, though, is that \eqref{eq:MultiODE2} can be written in the form \eqref{eq:ODE2} in a multitude of ways -- indeed, even the coefficients $a^{\mathbf{\ell}}_k$ are not unique. This can be seen in the simplest nontrivial case, $d=2$ and $j=2$:
\begin{eqnarray*}
  &&y_1'=a^{1,1}_1 y_1^2+(a^{1,2}_1+a^{2,1}_1)y_1y_2+a^{2,2}_1y_2^2,\\
  &&y_2'=a^{1,1}_2 y_1^2+(a^{1,2}_2+a^{2,1}_2)y_1y_2+a^{2,2}_2y_2^2.
\end{eqnarray*}
Therefore
\begin{displaymath}
  A({\bf y})=\left[
  \begin{array}{cc}
    a^{1,1}_1 y_1+\beta_{1,1}y_2 & \beta_{1,2}y_1+a^{2,2}_1y_2\\
    a^{1,1}_2y_1+\beta_{2,1}y_2 & \beta_{2,2}y_1+a^{2,2}_2y_2
  \end{array}
\right]\!,
\end{displaymath}
where
\begin{equation}
  \label{eq:ExtraEquality}
  \beta_{1,1}+\beta_{1,2}=a^{1,2}_1+a^{2,1}_1,\qquad  \beta_{2,1}+\beta_{2,2}=a^{1,2}_2+a^{2,1}_2.
\end{equation}
We deduce that in this case the graph-Laplacian conditions (which must hold for all ${\bf y}\succeq{\bf 0}$) are
\begin{eqnarray*}
  &&a^{1,1}_1,a^{2,2}_2,\beta_{1,1},\beta_{2,2}\leq0,\\
  &&a^{1,1}_1+a^{1,1}_2=a^{2,2}_1+a^{2,2}_2=\beta_{1,1}+\beta_{2,1}=\beta_{1,2}+\beta_{2,2}=0.
\end{eqnarray*}
Six equalities (inclusive of \eqref{eq:ExtraEquality}) and four inequalities for eight variables: impossible in some configurations, while other configurations lead to an infinity of solutions. 

Henceforth we let ${\bf e}_i$ stand for the $i$th unit vector. 

\begin{theorem}
  \label{thm:GraphLaplacian}
  The ODE system 
  \begin{equation}
    \label{eq:PolyODE1}
    y_k'=\sum_{\stackrel{\scriptstyle \ell_1+\cdots+\ell_d=2}{\ell_1,\ldots,\ell_d\geq0}} a^{\bf \ell}_k y_1^{\ell_1}y_2^{\ell_2}\cdots y_d^{\ell_d},\qquad k=1,\ldots,d
  \end{equation}
  admits the graph Laplacian representation \eqref{eq:ODE2} subject to the assumptions
  \begin{eqnarray}
    \label{eq:Ass1}
    &&a_k^{2{\bf e}_k}\leq0,\qquad a_k^{2{\bf e}_i}\geq0,\quad k,i=1,\ldots,d,\quad i\neq k,\\
    \label{eq:Ass2}
    &&a_k^{{\bf e}_i+{\bf e}_k}\leq0,\quad a_k^{{\bf e}_i+{\bf e}_j}\geq0,\qquad k,i,j=1,\ldots,d,\quad i\neq j,\quad k\neq i,j,\hspace*{20pt}\\
    \label{eq:Ass3}
    &&\sum_{k=1}^d a_k^{{\bf e}_k+{\bf e}_i}=0,\qquad i=1,\ldots,d.
  \end{eqnarray}
\end{theorem}

\begin{proof}
  We prove the theorem by constructing explicitly a graph Laplacian $A({\bf y})$, letting 
  \begin{equation}
    \label{eq:Constructive}
    A_{k,\ell}({\bf y})=a_k^{2{\bf e}_\ell} y_\ell+a_k^{{\bf e}_\ell+{\bf e}_{\ell+1}} y_{\ell+1},\qquad k,\ell=1,\ldots,d \quad \pmod d.
  \end{equation}
  All that remains is to prove that $A({\bf y})$, as defined in \eqref{eq:Constructive}, is indeed a graph Laplacian. Thus, recalling that $y_1,\ldots,y_d\geq0$ and  that $k$ is computed {\em modulo\/} $d$, 
  \begin{displaymath}
    A_{k,k}({\bf y})=a_k^{2{\bf e}_k} y_k+a_k^{{\bf e}_k+{\bf e}_{k+1}}y_{k+1}\leq0
  \end{displaymath}
  because of \eqref{eq:Ass1} and \eqref{eq:Ass2}. These two conditions also imply that 
  \begin{displaymath}
    A_{k,\ell}({\bf y})=a_k^{2{\bf e}_\ell}y_\ell+a_k^{{\bf e}_\ell+{\bf e}_{\ell+1}}y_{\ell+1}\geq0,\qquad k\neq \ell.
  \end{displaymath}
  Finally, it follows from \eqref{eq:Ass3} that
  \begin{displaymath}
    \sum_{k=1}^d A_{k,\ell}({\bf y})=\left(\sum_{k=1}^d a_k^{2{\bf e}_\ell}\right) y_\ell +\left(\sum_{k=1}^d a_k^{{\bf e}_\ell+{\bf e}_{\ell+1}}\right)y_{\ell+1}=0
  \end{displaymath}
  and we are done. 
\end{proof}

As an example, we revisit \eqref{eq:Robertson}, focussing on the quadratic part. Now
\begin{displaymath}
  a_1^{{\bf e}_2+{\bf e}_3}=10^4,\quad a_2^{{\bf e}_2+{\bf e}_3}=-10^4,\quad a_2^{2{\bf e}_2}=-3\cdot10^7,\quad a_3^{2{\bf e}_2}=3\cdot10^7
\end{displaymath}
and the remaining coefficients are zero: it is easy to verify that the conditions of Theorem~\ref{thm:GraphLaplacian} are satisfied. The representation \eqref{eq:Constructive}, incidentally, corresponds to \eqref{eq:RobertsonLin}, the graph-Laplacian form of of the Robertson reaction.

In this paper we focus only on equations \eqref{eq:PolyODE1}. The situation is more subtle for higher-order equations. For example, consider the case $d=2$, $m=3$ and
\begin{eqnarray*}
  &&y_1'=-\alpha_2^{3,0}y_1^3+\alpha_1^{2,1}y_1^2y_2+\alpha_1^{1,2}y_1y_2^2+\alpha_1^{0,3}y_2^3,\\
  &&y_2'=\alpha_2^{3,0}y_1^3-\alpha_1^{2,1}y_1^2y_2-\alpha_1^{1,2}y_1y_2^2-\alpha_1^{0,3}y_2^3.
\end{eqnarray*}
The most general way of writing it in the form \eqref{eq:ODE2} is with the matrix
\begin{displaymath}
  A({\bf y})=\left[\!\!
  \begin{array}{cc}
    -\alpha_2^{3,0}y_1^2\!-\!\beta_{2,1}^{2,1}y_1y_2\!-\!\beta_{2,1}^{1,2}y_2^2 & (\alpha_1^{2,1}\!+\!\beta_{2,1}^{2,1})y_1^2\!+\!(\alpha_1^{1,2}\!+\!\beta_{2,1}^{1,2})y_1y_2\!+\!\alpha_1^{0,3}y_2^2\\[4pt]
    \alpha_2^{3,0}y_1^2\!+\!\beta_{2,1}^{2,1}y_1y_2\!+\!\beta_{2,1}^{1,2}y_2^2 & -(\alpha_1^{2,1}\!+\!\beta_{2,1}^{2,1})y_1^2\!-\!(\alpha_1^{1,2}\!+\!\beta_{2,1}^{1.2})y_1y_2\!-\!\alpha_1^{0,3}y_2^3
  \end{array}
  \!\!\right]\!,
\end{displaymath}
where $\beta_{2,1}^{2,1}$ and $\beta_{2,1}^{1,2}$ are constants. Clearly, to have a graph Laplacian for all ${\bf y}\succeq{\bf 0}$ we require $\alpha_1^{0,3},\alpha_2^{3,0}\geq0$ and the two parameters need to satisfy 
\begin{displaymath}
  \beta_{2,1}^{2,1}\geq\max\{0,-\alpha_1^{2,1}\},\qquad \beta_{2,1}^{1,2}\geq \max\{0,-\alpha_1^{1,2}\}.
\end{displaymath}
Note that it is possible for $\beta_{2,1}^{2,1}<0$, say, and yet $A_{2,1}\geq0$, provided that $\beta_{2,1}^{1,2}\geq0$ and $\beta_{2,1}^{2,1}\geq-2\sqrt{\alpha_2^{3,0}\beta_{2,1}^{1,2}}$. As an example, we can write
\begin{displaymath}
  y_1'=-y_1^3+y_1^2y_2+y_2^3,\qquad y_2'=y_1^3-y_1^2y_2-y_2^3
\end{displaymath}
in the form \eqref{eq:ODE2} with
\begin{displaymath}
  A({\bf y})=\left[
  \begin{array}{cc}
    -y_1^2 & y_1^2+y_2^2\\[2pt]
    y_1^2 & -(y_1^2+y_2^2)
  \end{array}
\right]
\end{displaymath}
but it can also be written as
\begin{displaymath}
  A({\bf y})=\left[
  \begin{array}{cc}
    -y_1^2+\frac12 y_1y_2-y_2^2 & \frac12 y_1^2+y_1y_2+y_2^2\\[2pt]
    y_1^2-\frac12 y_1y_2+y_2^2 & -\frac12 y_1^2-y_1y_2-y_2^2
  \end{array}
\right]=\left[
  \begin{array}{cc}
    -(y_1-y_2)^2 & \frac12 (y_1^2+y_2)^2\\[2pt]
    (y_1-y_2)^2 & -\frac12 (y_1^2+y_2)^2
  \end{array}
\right].
\end{displaymath}
Note that this cannot occur for quadratic equations \eqref{eq:PolyODE1} because, once $A_{k,\ell}({\bf y})$ is a multilinear function of ${\bf y}\succeq{\bf 0}$, it is a graph Laplacian only if all off-diagonal coefficients are nonnegative.

\subsection{Chemical reactions by the Law of Mass Action}
An important application are chemical reactions, where the rate of reaction is modelled by the Law of Mass Action.
Then the model is a first-order ODE with a multivariate polynomial for the right hand side, so it can be considered an important special case of our framework.
Suppose there are $N$ reactions, where the $j$-th reaction is written in the form
\begin{displaymath}
  r_{j,1}G_1+ r_{j,2}G_2+ \ldots +r_{j,M}G_M \;\; 
	\xrightarrow[]{k_j}
	 \;\; q_{j,1}G_1+q_{j,2}G_2+ \ldots +q_{j,M}G_M ,
\end{displaymath}
$j=1\ldots,N$. 
Here $r_{i,j}, q_{i,j}$  are  integer coefficients, $G_i, \ i = 1 ,2, \ldots ,M$ are symbols for the chemical species, $y_i$ denotes the concentration of species $i$, and $k_j$ is the rate constant. 
The model is the ODE
\begin{equation}\label{eq:ChemiReac1}
  \yy' = S {\bf p}, \qquad \yy(0)=\yy_0
\end{equation}
where $\yy \in \BB{R}^M,  S\in \BB{R}^{M\times N}, \ {\bf p}\in \BB{R}^N$, and $  S_{i,j}= q_{i,j}-r_{i,j}$ is the matrix of {\em stoichiometric vectors\/}, while
\begin{displaymath} 
	p_j=k_j\prod_{i=1}^M y_i^{r_{i,j}}
\end{displaymath}
is the Law of Mass Action to model the rates of reaction.
This is a nonlinear and autonomous differential equation (it would be non-autonomous if the rates $k_j=k_j(t)$ were time-varying, for example to model fluctuating temperatures).
The following theorem shows this model can always be written in the form 
\begin{equation}
\label{eq:Linearised1}
\yy' = \mathcal{L}(\yy) \yy,
\end{equation}
 where the matrix $\mathcal{L}(\yy)$ has the same pattern of signs as a Laplacian, i.e. off-diagonal entries are nonnegative, and negative entries can only appear on the diagonal.

\begin{theorem}
\label{thm:LawMassAction:Laplacian}
 The nonlinear ODE \eqref{eq:ChemiReac1} can be written in the quasi-linearised form \eqref{eq:Linearised1} where the negative elements of the matrix $\mathcal{L}(\yy)$ only appear on the diagonal.
\end{theorem}
\begin{proof}
  We assume that $q_{i,j},r_{i,j}$ are non-negative integers, $k_j$ is a non-negative real number and $y_j\geq 0$. Then, a negative coefficient could only appear in the stoichiometric matrix $S_{i,j}$ if $r_{i,j}\geq 1$, and this happens in the equation for $y_i'$. 
  Since we have 
	$
	  p_j=k_j\!\left(\prod_{k=1, k\neq i}^M y_k^{r_{k,j}}\right)\!  y_i^{r_{i,j}}
	$
with  $r_{i,j}\geq 1$, we may allocate this term to the diagonal of the matrix $\mathcal{L}(\yy)$.
All other components where  $r_{i,j}=0$ in the right hand side of the equation for $y_i'$ have positive coefficients and can be allocated outside the diagonal.
\end{proof}

\paragraph{Remark.}
Note that the matrix $\mathcal{L}(\yy)$ of \eqref{eq:Linearised1} need not be unique, as we previously showed by the example of the Robertsons reaction in \eqref{eq:RobertsonLin}.
The theorem shows that we may form $\mathcal{L}(\yy)$ so that it has the right pattern of signs to be a graph Laplacian.
Similarly to the remarks following Proposition~\ref{prop:GraphLaplacians},
this ensures positivity of the solutions, and the point we are making here  is that the new numerical methods proposed in this paper can be applied, via  \eqref{eq:Linearised1}, to this big class of important applications.
The only difference between \eqref{eq:Linearised1} and the primary focus of this paper in \eqref{eq:ODE2}, is that in \eqref{eq:ODE2} we additionally assume that ${\bf 1}^\top$ is in the left null space of $A(\yy)$, but that does not prevent us from applying the numerical schemes proposed in this paper, and they will preserve positivity as required.
(Although there may be issues with other conservation laws, as we show in the autonomous oscillations example \eqref{eq:Oscillation:model}, and our atmospheric chemistry example \eqref{eq:stratospheric}.)

\setcounter{equation}{0}
\setcounter{figure}{0}
\section{Numerical experiments}

In this section we present some numerical experiments to illustrate the performance of the new methods on a number of examples from the literature. We denote:

\begin{itemize}
	\item ES2: The symmetric second order 3-exponential splitting method \eqref{eq:smoothing} or \eqref{eq:smoothing:nonaut};
	\item EM1: The first order 1-exponential Magnus integrator \eqref{eq:ExpoRep} or \eqref{eq:1.3t};
	\item EM2: The second order 2-exponential Magnus integrator \eqref{eq:ExpoRepb} or \eqref{eq:1.3bt};
	\item EM3: The third order 7-exponential Magnus integrator \eqref{eq:CF3} or \eqref{eq:CF3t}.
\end{itemize}
We will also consider, for comparison, the following more conventional numerical solvers:
\begin{itemize}
	\item Euler: The first-order explicit Euler method;
	\item RK4: The 4-stage fourth-order explicit RK method (as a reference method to compare);
	\item ROS4: The 4-stage fourth-order Rosenbrock method with coefficients used by default in \cite{hairer10sod}.
	\item MP2: The second order Modified Patankar method.
\end{itemize}

\subsection{Example 1: The SIDARTHE mathematical model}

We first consider a generalised SIR model (SIDARTHE) that has been used to model the evolution of the Cov-SARS-2 epidemic in Italy  \cite{giordano20mtc}. That model can also be used for any other country with appropriate data or it can be even extended e.g.\ to age-dependent variables.

 The SIDARTHE dynamical system \cite{giordano20mtc} consists of eight ordinary differential equations, describing the evolution of the population in each stage over time. The equations can be written in the form
\begin{displaymath}
  {\bf y}' = A({\bf b}(t),{\bf y}){\bf y}, \qquad {\bf y}(0)={\bf y}_0\in\BB{R}^8,
\end{displaymath}
where ${\bf b}:\BB{R}\rightarrow \BB{R}^{15}$ is a vector function depending on 15 time-dependent parameters. The vector ${\bf b}$ was taken as a piecewise constant function, and the authors estimate the model parameters based on data from 20 February 2020 (day 1) to 5 April 2020 (day 46) and show the impact of progressive restrictions on the spread of the epidemic. For example, ${\bf b}$ is constant from day 1 to 4 (with a value of $R_0=2.28$), and changes to new constant values for the period 4 to 12 (with a value of $R_0=1.66$), and so on.

Notice that since the vector field is not a smooth function (it is piecewise constant) the numerical methods deteriorate down to order one. However, a more realistic model should consider  ${\bf b}(t)$  as a smooth time-dependent function, and in this case the order of the methods is recovered.

For simplicity, we take the same initial values for ${\bf b}$ and the same initial conditions ${\bf y}_0$ as in \cite{giordano20mtc}, but we take  ${\bf b}$ constant for a longer period, from day 1 to 20.

We observed that the model is very sensitive to the parameter associated to the first component of ${\bf b}$, $b_1=\alpha$. That parameter was taken initially as $\alpha=0.57$, and we have analysed the solution for the first component of ${\bf y}$ (i.e. ${y}_1(t)=S(t)$, the susceptible (uninfected) population at day 20) for different values of  $\alpha=0.57\cdot r$ with $r\in[1,2]$. The results are shown in Figure~\ref{fig:SIDARTHE0}.

\begin{figure}[hbt]%
\begin{center}%
    \hspace*{-0.13cm}
    \includegraphics[width=8cm]{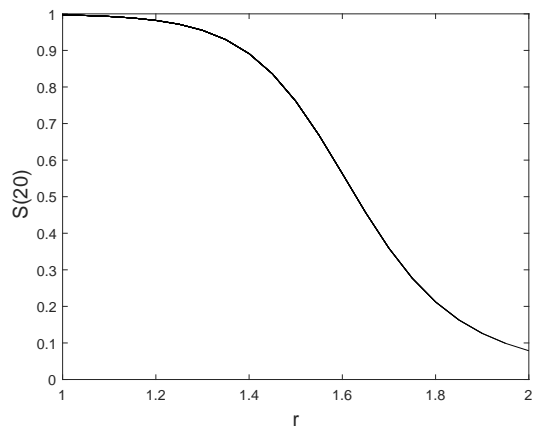} 
   \caption{  Solution for the susceptible population at $t=20$ for different values of the parameter $r$ where $\alpha=0.57\cdot r$.} 
 \end{center}\label{fig:SIDARTHE0}
\end{figure}

Next, we take $r=1.5$, corresponding to a moderately stiff problem ($S(20)$ still has not dramatically decreased)  and we compute the 2-norm error of the solution ${\bf y}(20)$ versus the time step for the new methods as well as for the explicit Euler method that was used in \cite{giordano20mtc}. 
The results are displayed in Figure~\ref{fig:SIDARTHE} (left) 
where the order of the methods is clearly visible from the slopes of the curves.

The new methods require to compute matrix exponentials and this can be computationally costly in some cases. It is thus interesting to study if it is possible to replace the exact exponential of matrices by cheaper approximations while still preserving positivity.

This is not a very stiff problem and we have repeated the same numerical experiments while replacing each exponential by the second-order diagonal Pad\'e approximation. In order to preserve positivity, we proceed as follows, given $A=\tilde A+a^*I$ where $\tilde A \succeq O$, we consider the following approximation to the exponential
\begin{displaymath}
    \e^{tA} = \e^{ta^*} \e^{t\tilde A} \simeq \e^{ta^*}  \,\frac{1+\frac12 t\tilde A}{1-\frac12 t\tilde A} .
\end{displaymath} 
Note that, since ${\bf 1}^\top\tilde A = -a^*$, we have 
\begin{displaymath}
{\bf 1}^\top \e^{ta^*}  \ 
		\frac{1+\frac12 t\tilde A}{1-\frac12 t\tilde A} =
		\e^{ta^*} \frac{1-\frac12 ta^*}{1+\frac12 ta^*}  \neq 1
\end{displaymath}
and mass is not preserved. This can be fixed, for example, if we also approximate the scalar function $\e^{ta^*}$ by the second-order diagonal Pad\'e approximation, so
\begin{displaymath}
{\bf 1}^\top \frac{1+\frac12 ta^*}{1-\frac12 ta^*}  \ 
		\frac{1+\frac12 t\tilde A}{1-\frac12 t\tilde A} =1
\end{displaymath}
and this approach preserves norm and positivity in the stability region.

The results are shown in Figure~\ref{fig:SIDARTHE} (right). We observe that the schemes maintain their accuracy while being considerably cheaper. The third-order method EM3 exhibits  second order accuracy (due to the second order Pad\'e approximation) but this occurs only at higher accuracies.

{{}{For clarity in the presentation, the results for MP2 are not shown but, as expected they are slightly worse but close to the results given by EM2.}

Unfortunately, this is not the case if we repeat the numerical experiment with the very stiff problem of  Robertson's reaction. Once higher-order approximations to the exponential are used, positivity is not guaranteed. Not all higher-order Pad\'e approximations preserve positivity, unlike the second order one, and this deserves further investigation.

\begin{figure}[ht]%
\begin{center}%
    \hspace*{-0.33cm}
    \includegraphics[width=6.4cm]{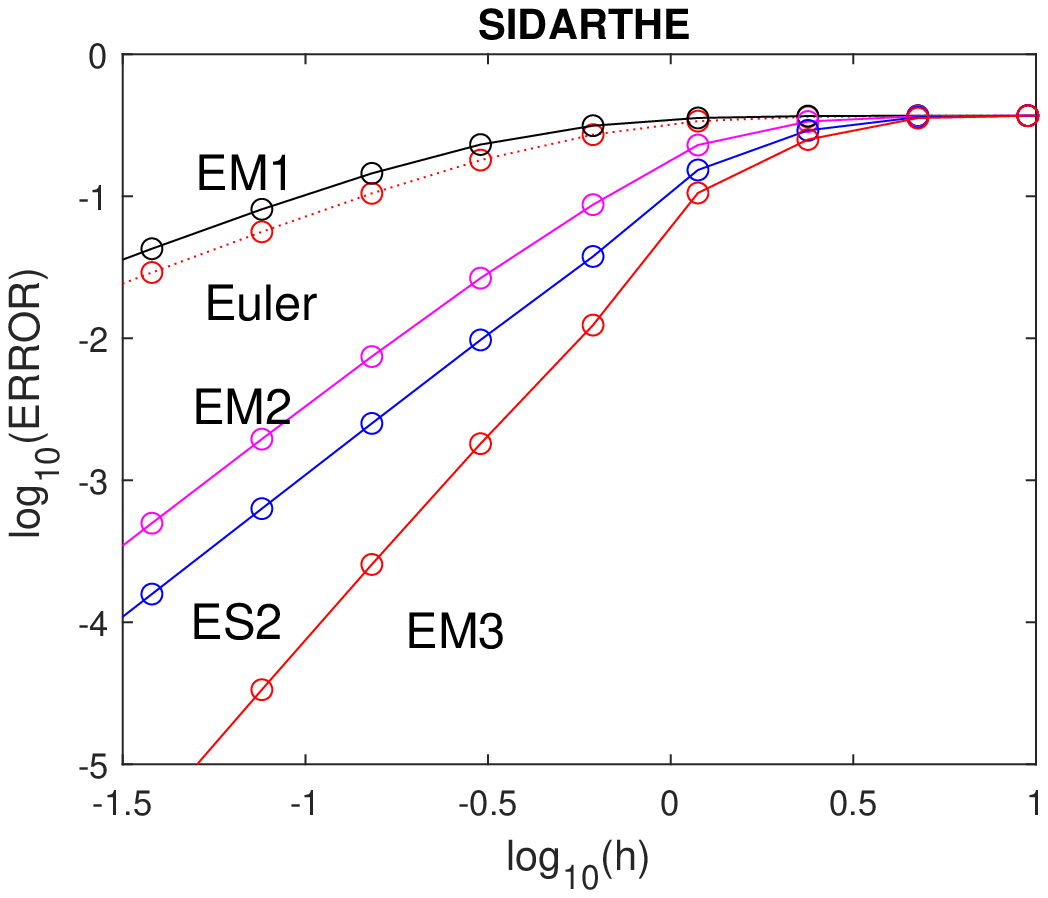} 
    \includegraphics[width=6.4cm]{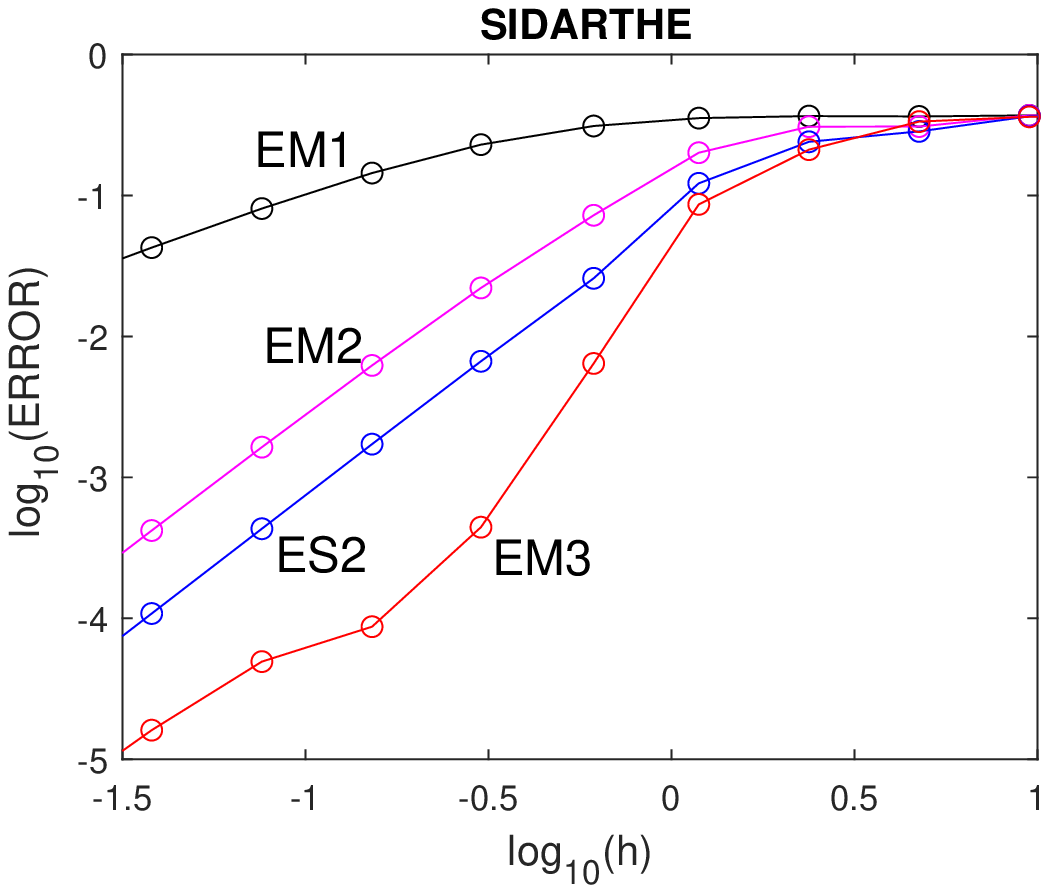} \\ 
    \caption{The 2-norm error of the solution of the SIDARTHE mathematical model at the final time versus the time step in double logarithmic scale: (left) the new methods compute the exponential of matrices to round-off accuracy (this, of course, is irrelevant to the explicit Euler method); and (right) the same new methods with each exponential  replaced by the second order Pad\'e approximation.}
		\label{fig:SIDARTHE}
\end{center}
\end{figure}

%

\subsection{Example 2: Robertson's reaction.} Let us now consider the Robertson's reaction written in the form \eqref{eq:RobertsonLin}
with initial conditions ${\bf y}_0=[1,0,0]^\top$ and time interval $t\in[0,0.3]$ as in  \cite{hairer10sod} (p.~57). We numerically solve the problem repeatedly using different values for the time step and compute the 2-norm error of the solution at the final time. Here, we compare with the `exact' solution that is computed numerically with sufficiently high accuracy. 

Notice that this is a very stiff problem that turns into a non-stiff problem if one applies an appropriate time transformation which can be integrated with a constant time step (in the fictitious time) by methods for non-stiff problems. This is basically the case studied in \cite{burchard03hoc} with time step $h_n=1.8^n\times h_0$ and initial time step $h_0=10^{-6}$ that allows to integrate for the interval $t\in[0,10^{11}]$ with a very small number of time steps, but the details in the reaction at the very beginning can be lost.

Figure~\ref{fig:Robertson1a} (left) shows the error  versus the time step in double logarithmic scale. The implicit Rosenbrock method, ROS4, outperforms the explicit RK methods, Euler and RK4, but also turns unstable for moderate values of the time step (and does not preserve positivity) while the new exponential methods preserve positivity and are unconditionally stable (the third order method, EM3, preserves positivity for all time steps considered). Note the relatively high accuracy provided by the new schemes even when considering large time steps. The best method among the proposed schemes  depends on the desired accuracy where the computational cost has to be taken into account.

{{}{As in the previous example, the results for MP2, not shown, are slightly worse but close to the results given by EM2.}

We have repeated the same numerical experiments using only the new exponential methods, but applied to the equations as given in \eqref{eq:RobertsonLinNO}, i.e.\ the same problem but written in a different way such that the matrix is no longer graph Laplacian. Figure~\ref{fig:Robertson1a} (right) shows the results obtained. We filled a relevant circle when, during the numerical integration, a negative solution was obtained on any of the components. For small time steps the performance is quite similar (and the performance for EM1 is actually somewhat better) but the errors grow faster for large time steps (lower accuracies) and, even worse, negative solutions do occur.

\begin{figure}[htb]%
\begin{center}%
    \hspace*{-0.33cm}
    \includegraphics[width=6.4cm]{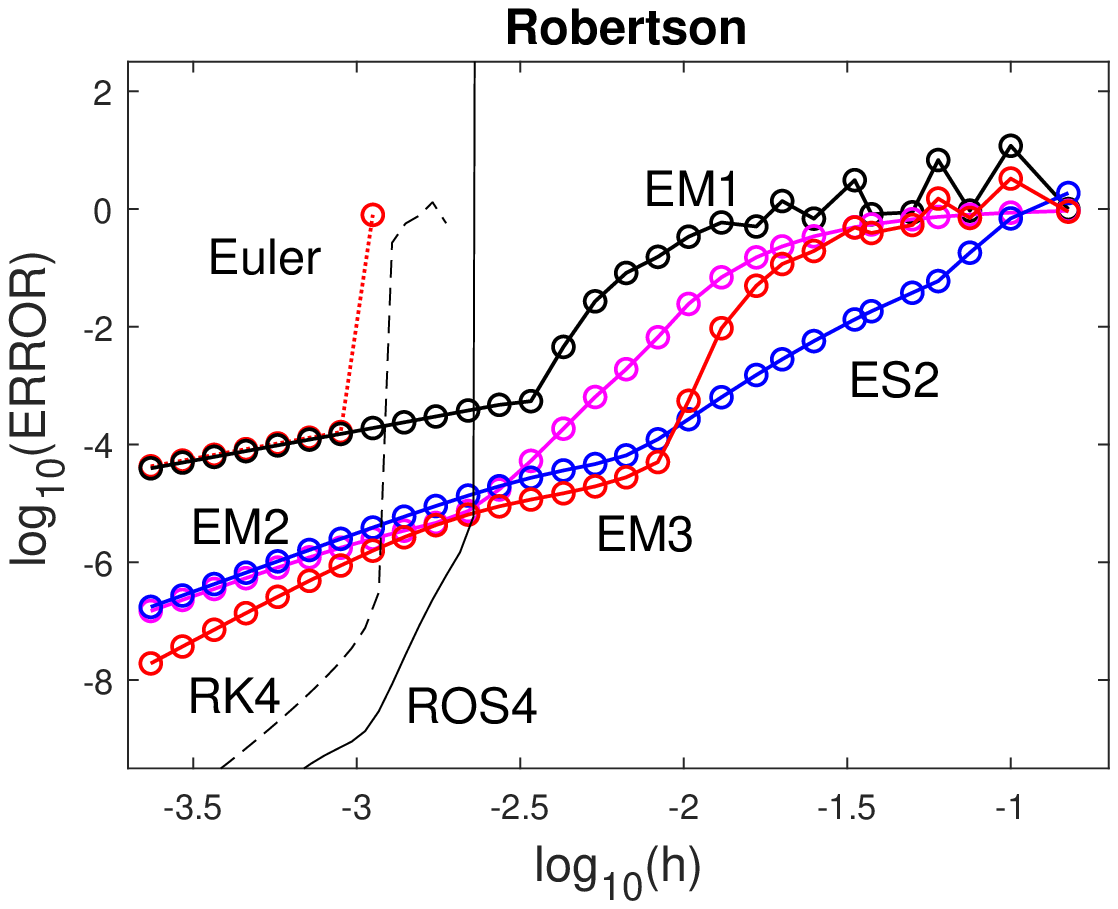} 
    \includegraphics[width=6.4cm]{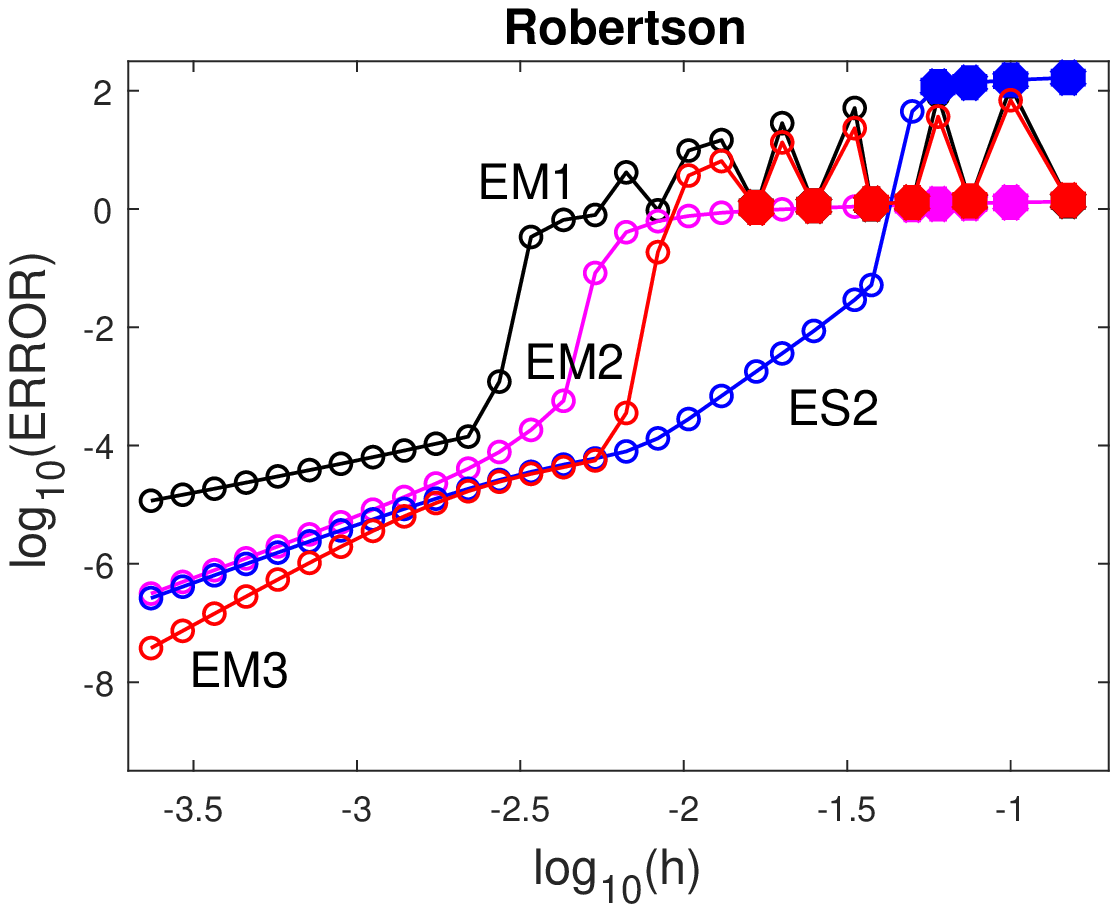} 
    \caption{The 2-norm error of the solution of the Robertson's reaction at the final time versus the time step in double logarithmic scale: (left) the new methods are used to solve \eqref{eq:RobertsonLin} where the matrix is graph Laplacian (the results of the standard methods Euler, RK4 and ROS4 are also included); and (right) the same new methods applied to solve the same problem, but written  in the form \eqref{eq:RobertsonLinNO}, where the matrix is no longer a graph Laplacian (a filled circle indicates that a negative solution on any of the components has been obtained in the course of the time integration).}
\end{center}\label{fig:Robertson1a} 
\end{figure}

\subsection{Example 3: The stratospheric reaction}

Let us consider the basic stratospheric reaction mechanism studied in \cite{sandu01pni} that involves six species
\[
   {\bf y}=[\,[O^{1D}],[O],[O_3],[O_2],[NO],[NO_2]\,]^\top=[y_1,\ldots,y_6]^\top
\]
and whose model to obtain the evolution of the concentrations is given by the system of ODEs
\begin{Eqnarray}
 \nonumber
 y_1'&=&k_5y_3-k_6y_1-k_7y_1y_3 \\
 \nonumber
 y_2'&=& 2k_1y_4- k_2y_2y_4+k_3y_3-k_4y_2y_3+k_6y_1-k_9y_2y_6+k_{10}y_6 \\
 \label{eq:stratospheric}
 y_3'&=&k_2y_2y_4-k_3y_3-k_4y_2y_3-k_5y_3-k_7y_1y_3-k_8y_3y_5 \\
 \nonumber
 y_4'&=& -k_1y_4-k_2y_2y_4+k_3y_3+2k_4y_2y_3+k_5y_3+2 k_7y_1y_3+ k_8y_3y_5+k_9y_2y_6  \\
 \nonumber
 y_5'&=& -k_8y_3y_5+k_9y_2y_6+ k_{10}y_6 \\
 \nonumber
 y_6'&=&  k_8y_3y_5 -k_9y_2y_6 -k_{10}y_6  
\end{Eqnarray}
with
\[
\begin{array} {lll}
 k_1=2.643\cdot 10^{-10}\sigma^3(t), &
 k_2=8.018\cdot 10^{-17},  &
 k_3=6.120\cdot 10^{-4}\sigma(t),
  \\[3pt]
 k_4=1.576\cdot 10^{-15}, & k_5=1.070\cdot 10^{-3}\sigma^2(t),&
  k_6=7.110\cdot 10^{-11}, \\[3pt]
k_7=1.200\cdot 10^{-10},  &
k_8=6.062\cdot 10^{-15}, & k_9=1.069\cdot 10^{-11},\\[3pt]
\!\!k_{10}=1.289\cdot 10^{-2}\sigma(t),& & 
\end{array} 
\]
where
\[
   \sigma(t)=\left\{ \begin{array}{ll}
	\frac12+\frac12\cos\!\left( 
	\pi \left|\frac{2T_L-T_R-T_S}{T_S-T_R} \right|\frac{2T_L-T_R-T_S}{T_S-T_R} \right)
	\qquad &  \mbox{if} \ T_R\leq T_L \leq T_S \\
	0 &  \mbox{otherwise} .
	\end{array}  \right.
\]
The time is measured in seconds and it is taken as
\[
 T_L=\left(\frac{t}{3600} \right)\bmod 24, \qquad T_R=4.5, \qquad T_S=19.5.
\]
The initial time is considered at noon, $t_0=12\times 3600$, and it is integrated for three full days, until $t_f=t_0+72\times 3600$ with initial conditions given by
\[
 {\bf y}_0=[9.906\cdot 10^{1}, 6.624\cdot 10^{8}, 5.326\cdot 10^{11}, 
1.697\cdot 10^{16}, 8.725\cdot 10^{8}, 2.240\cdot 10^{8}]^\top.
\]

This is a non-autonomous systems that can be written in the form
\[
  {\bf y}'=A(t,{\bf y}) {\bf y}
\]
with $A(t,{\bf y})$ an explicitly time-dependent graph Laplacian matrix.
We can write the vector field in terms of the production and destruction parts
\[
   A(t,{\bf y}) {\bf y}=P(t,{\bf y})-D(t,{\bf y}) {\bf y}
\]
where $P(t,{\bf y}),D(t,{\bf y}) {\bf y}$ are non-negative. While the diagonal matrix $D$ is unique in this case, we can write 
\[
  P(t,{\bf y})=A_P(t,{\bf y}) {\bf y}
\] 
in many different ways for the matrix $A_P$. 
We have considered the following choice (other choices of $A_P$ can be considered) for $A$,
{\small\begin{displaymath}
\!\left[\begin{array} {cccccc}
 \!\!-(k_6+k_7y_3)\!\!& 0 & k_5  & 0  & 0  & 0  \\[2pt]
  k_6 &\!\!- (k_2y_4+k_4y_3+k_9y_6) \!\!& k_3 & 2k_1 & 0 & k_{10} \\[2pt]
  0  & \frac13k_2y_4 &  
	-\gamma
	&  \frac23k_2y_2  &  0  &  0 \\[2pt]
  \frac12 k_7y_3  & k_4y_3 + \frac12k_9y_6 & 
	\gamma + \frac12k_7y_1
	& 
	\!\!-(k_1+k_2y_2)\!\!  &  0  & \frac12k_9y_2 \\[2pt]
  0  & 0 &  0 & 0 & -k_8y_3  &  \!\!k_{10}+k_9y_2\!\!\\[2pt]
  0  & 0 &  0 & 0 &  k_8y_3  &  \!\!-(k_{10}+k_9y_2)\!\!
\end{array} \right]\!
\end{displaymath}}
with $\gamma=k_3+k_5+k_4y_2+ k_7y_1+k_8y_5$.

This problem has two linear mass conservation laws, 
{{}{
the number of atoms of oxygen and nitrogen, respectively.
 }
Given 
\[
  {\bf w}_1=[1,1,3,2,1,2]^\top, \qquad
  {\bf w}_2=[0,0,0,0,1,1]^\top
\]
it is true that
\[
  {\bf w}_1^\top A(t,{\bf y}) {\bf y}=
  {\bf w}_2^\top A(t,{\bf y}) {\bf y}={\bf 0}.
\]

Unfortunately, it is impossible to find a matrix $A_P$ such that 
\[
  {\bf w}_1^\top A(t,{\bf y})={\bf w}_2^\top A(t,{\bf y})={\bf 0},
\]
and both mass conservations cannot be simultaneously preserved by our schemes. We have to decide how to choose $A_P$ to optimise the performance of our methods: this is typical to geometric numerical integration of differential equations with multiple invariants. 

For this particular choice we have 
\[
 {\bf w}_2^\top A(t,{\bf y})={\bf 0}, \qquad \mbox{but} \qquad 
 {\bf w}_1^\top A(t,{\bf y})\neq {\bf 0},
\]
and then, in general,  ${\bf w}_1^\top {\bf y}(t)\neq \mbox{const}$. However, a good choice for $A_P$ can provide solutions where this quantity is preserved to very high accuracy.

We have observed that $y_1,y_2$ and $y_5$ take very small, but positive, values (say $10^{-200}$ or smaller) along the integration (standard methods usually provide negative values). In that case, measuring  relative error is not appropriate for these components.

Figure~\ref{fig:Stratospheric1} shows the evolution of the concentration of the different species in a logarithmic scale. Negative values in this plot correspond to having  no particles.

\begin{figure}[hbt]%
\begin{center}%
    \hspace*{-0.13cm}
    \includegraphics[width=10cm]{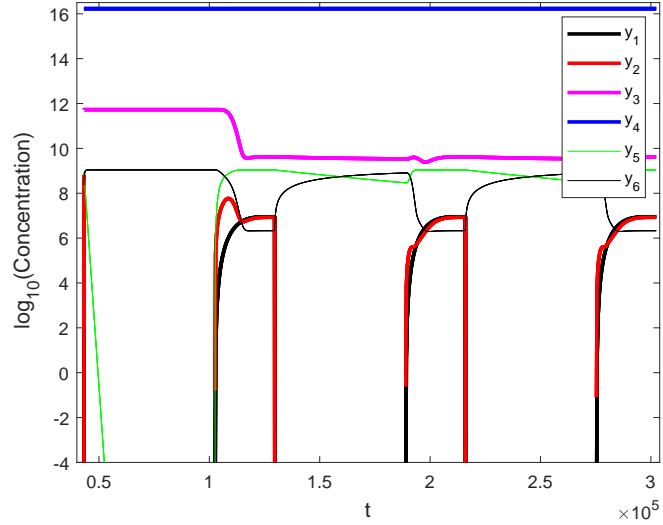} 
   \caption{  Solution for the concentrations for the stratospheric reaction in a logarithmic scale.
            }
\label{fig:Stratospheric1}
\end{center}
\end{figure}

\begin{figure}[ht]%
\begin{center}%
    \hspace*{-0.33cm}
    \includegraphics[width=6.4cm]{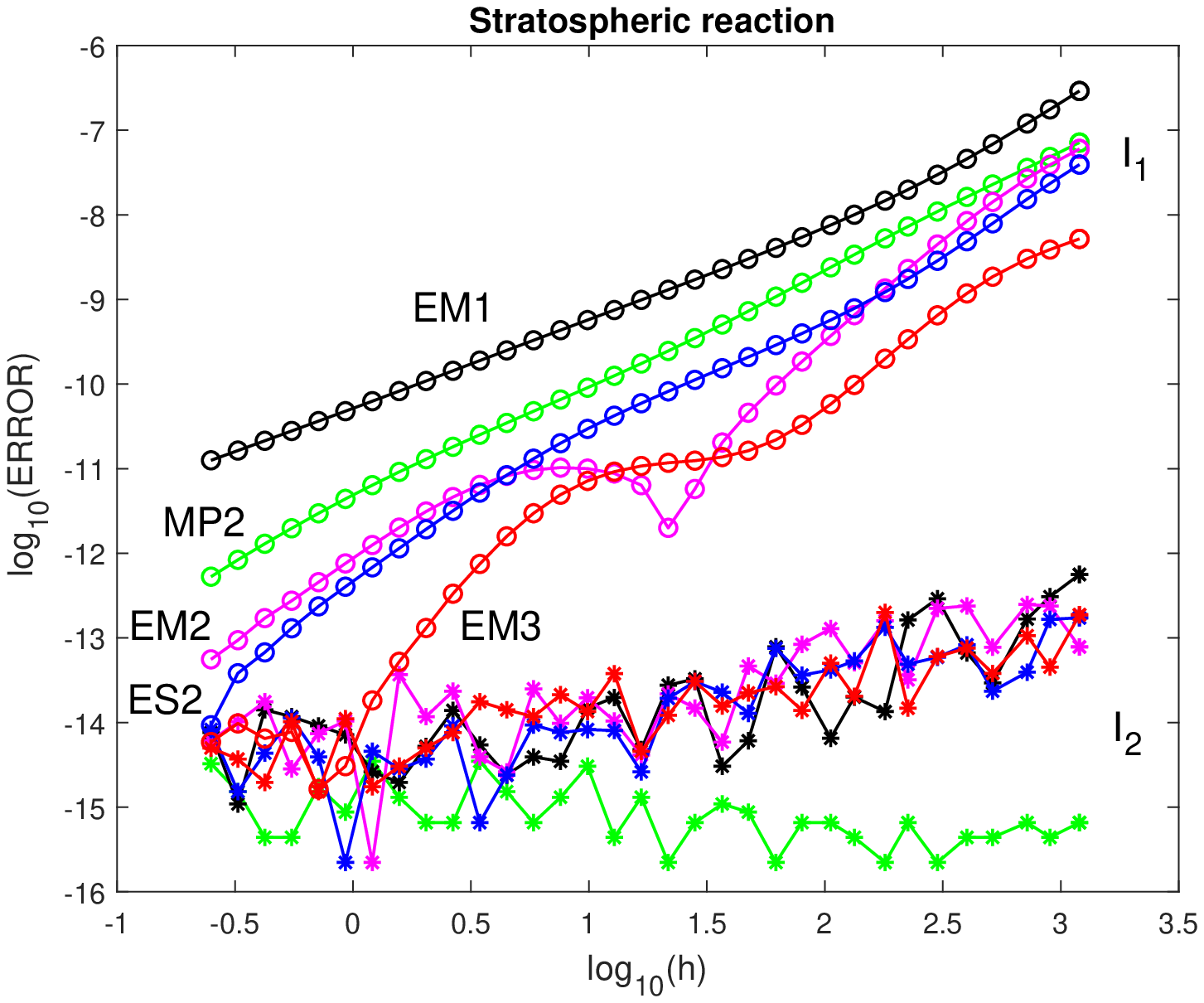} 
    \includegraphics[width=6.4cm]{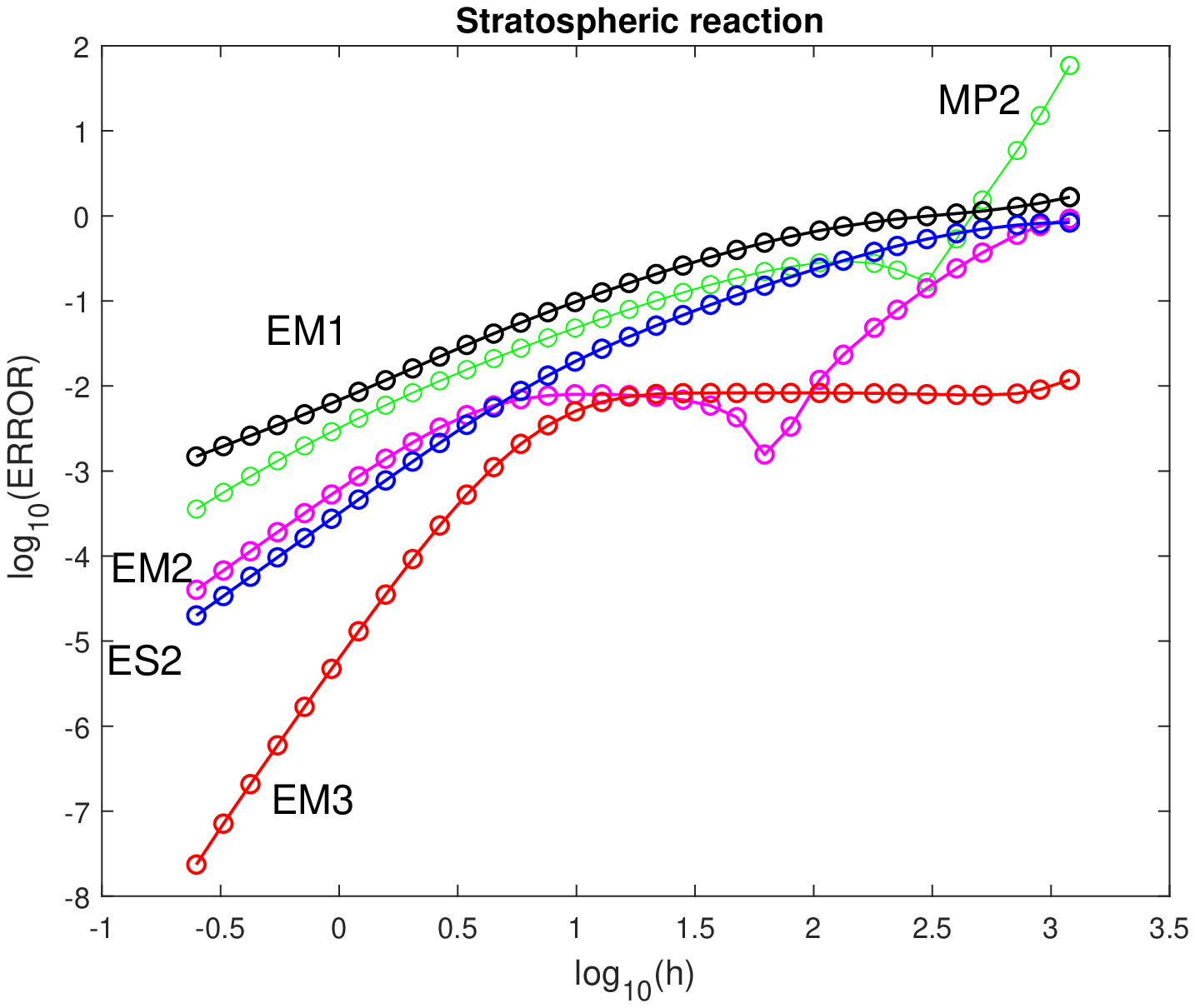} \\ 
    \caption{
Left: Error in the preserved quantities $I_1$ and $I_2$ for the stratospheric reaction model. Right:  The 2-norm error of the numerical solutions for the stratospheric reaction model for the components $(y_3,y_4,y_5,y_6)$ at the final time $t_f=t_0+3600$ (one hour) versus the time step in double logarithmic scale.}
		\label{fig:Stratospheric2}
\end{center}
\end{figure}

We have repeated the numerical experiments, integrating for just one hour (instead of 72 hours) and measured the two-norm relative error for the vector with components $\tilde{\bf y}=[y_3,y_4,y_5,y_6]$ since at the final time $y_1$ and $y_2$ vanish. The reference solution is obtained numerically using the third-order method and a sufficiently small time step. Figure~\ref{fig:Stratospheric2} (right panel) shows the results obtained where we can observe the order of convergence of each method for this non-autonomous problem.
Figure~\ref{fig:Stratospheric2} (left panel) shows 
the error in the preservation of the quantities $I_1={\bf w}_1^\top {\bf y}(t_f)$ (curves with circles) and $I_2={\bf w}_2^\top {\bf y}(t_f)$ (curves with stars). 
{Remarkably, the error committed for $I_2$ is orders of magnitude smaller than the error in the actual solution, as seen the left panel in Fig.~\ref{fig:Stratospheric2}!} 

{{}{
We observe that MP2, as in the previous examples, provide slightly worse results than EM2 when accurate results are desired, but the error considerably grows for large time steps (the approximation to the exponential in this case is not accurate). Surprisingly, it provides more accurate results in the exact preservation of $I_2$ and for all time steps positivity was preserved even if this property was not guaranteed for the method MP2 since $A$ is not graph Laplacian. Were one to prove that the matrix $A$ has no eigenvalues with positive real part then $I-tA$ would be an $M$-matrix and positivity would be guaranteed, and this deserves further investigation.  
 }


\subsection{Example 4: The MAPK cascade}

Finally, we consider the model of \cite{hadavc2017minimal} (Table 3, Fig 3, equations (12)-(17)), which is closely related to the MAPK cascade, given in \eqref{eq:Oscillation:model} with values therein for the parameters and initial conditions. The solution for each component is shown in the left panel of Figure~\ref{fig:MAPK}  for the time interval $t\in[0,200]$ (the initial conditions clearly identify each curve) where we observe that, after a transition period, the solution turns nearly periodic. Next, we have numerically solved the problem  for $\alpha=1$ using the new exponential methods using different values of the time step and measured the two-norm relative error in the vector solution at the final time. The right panel of Figure~\ref{fig:MAPK} shows the results obtained.

\begin{figure}[ht]%
\begin{center}%
    \hspace*{-0.33cm}
    \includegraphics[width=6.4cm]{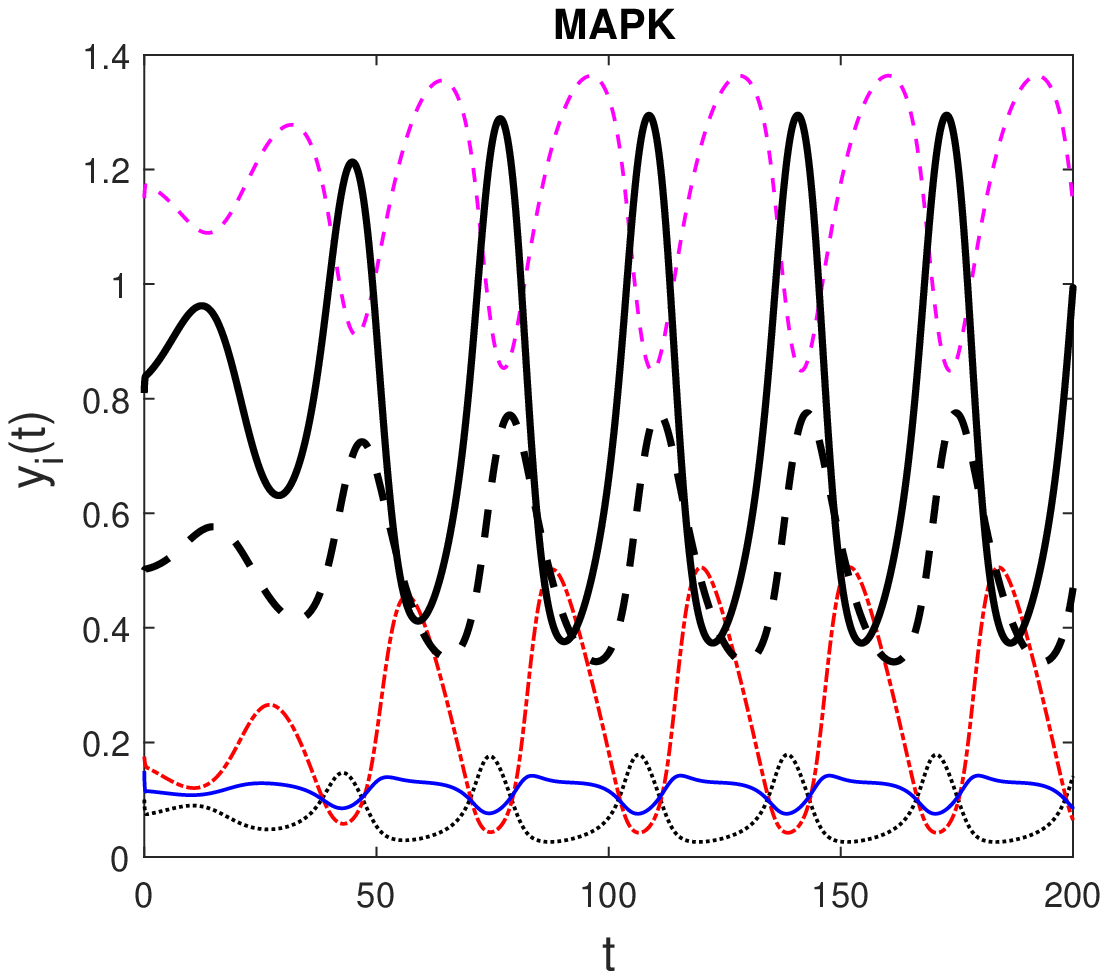} 
    \includegraphics[width=6.4cm]{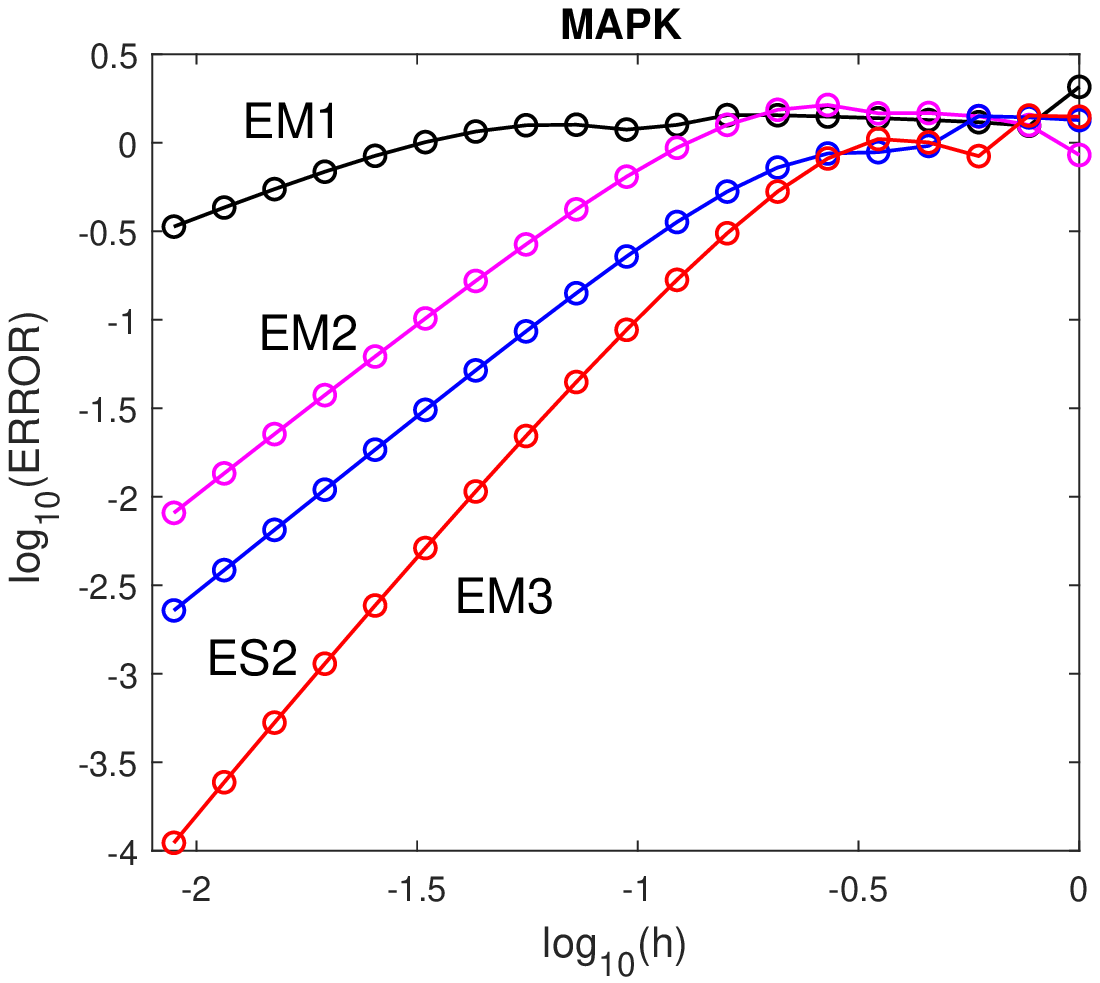} \\ 
    \caption{Left: Solution of the MAPK cascade, given in \eqref{eq:Oscillation:model}, and right: two-norm relative error in the vector solution at the final time versus the time step in double logarithmic scale.}
		\label{fig:MAPK}
\end{center}
\end{figure}

\setcounter{equation}{0}
\setcounter{figure}{0}
\section{Conclusions}

{
Preservation of inequalities is considerably more challenging than the recovery of `equality invariants' under discretisation. Thus, while numerous geometric numerical integration algorithms present us with a wide range of highly effective means to recover conservation laws, often of crucial importance in applications, this is not the case with inequalities and, of particular importance to us, nonnegativity of solutions. The importance of the latter in applications is clear -- the number of chemical species cannot be negative, temperature cannot be less that $0^\circ$K, the 
{{}{number of infected people $I(t)$}}
cannot (sadly) be negative -- yet we cannot be assured that computed ODE solutions remain nonnegative in this setting unless the order is unacceptably low. 
As aforementioned, the subject has already received significant attention and led to the development of Patankar-type methods \cite{burchard03hoc,kopecz18ooc,patankar80nht,bertolazzi1996pac}.
 In this paper we have developed a framework allowing us to use higher-order methods in this setting. While this framework is by no means final and many challenges remain, it represents in our view useful contribution to a different kind of geometric numerical integration, one dealing with preservation of inequalities.

An outstanding challenge is to approximate the exponential of matrices by diagonal Pad\'e approximants or by other means (e.g.\ Krylov-subspace methods) to reduce the cost of the algorithms for large ODE systems while still preserving positivity. Another is to explore the scope of methods, like the commutator-free Magnus integrators \eqref{eq:CF3}, which almost preserve positivity and formulate `almost preservation' in more precise terms. 

Yet, perhaps the most interesting challenge is to explore the surprising success of `almost positivity-preserving' methods, e.g.\ the fourth-order commutator-free Magnus method, in the examples in this paper. Recall that classical ODE solvers that preserve positivity are restricted to order one \cite{bolley78cdl}, while in this paper we have introduced second-order positivity-preserving methods in the non-classical class of Magnus integrators, and other high-order methods have been introduced elsewhere, in particular modified Patankar methods. It is natural to formulate the conjecture that this is as much as can be done within the realm of such methods, but equally fascinating is the remarkable almost-preservation of positivity or mass (at any rate in the examples of this paper) by some higher-order methods. For example, Figure~\ref{fig:Stratospheric2} (left) is concerned with two  conservation laws in a stratospheric reaction: one is preserved correctly, up to roundoff error, while the other is preserved to much higher accuracy than the error committed (cf.\ Fig.~\ref{fig:Stratospheric2} right) in the solution itself. We look forward to an explanation.

\section*{Acknowledgments}%
The authors thank the Isaac Newton Institute for Mathematical Sciences for support and hospitality during the programme ``Geometry, compatibility and structure preservation in computational differential equations" when work on this paper was undertaken. This work was supported by EPSRC grant  EP/R014604/1. S.B. has been supported by project PID2019-104927GB-C21 (AEI/FEDER, UE).

\bibliographystyle{plain}
\bibliography{Draft_GraphLaplacian_M2AM_ACCEPTED}

\end{document}